\documentclass[a4paper]{article}
\usepackage[utf8]{inputenc}
\usepackage{amssymb,amsmath,amsthm}
\usepackage{hyperref}
\usepackage{authblk}
\usepackage{a4wide}

    \newtheorem{observation}{Observation}

\begin{document}
    \newcommand{\fixE}[1]{#1}
    \newcommand{\fixK}[1]{#1}

    \renewcommand{\u}{{\bf u}}
    \newcommand{\A}{{\mathcal{A}}}
    \renewcommand{\L}{{\mathcal{L}}}

    \newcommand{\N}{{\mathbb N}}
    \newcommand{\R}{{\mathbb R}}
    \newcommand{\C}{{\mathbb C}}
    \newcommand{\Z}{{\mathbb Z}}
    \newcommand{\Zp}{{\mathbb Z}^{+}}
    \newcommand{\Zpz}{{\mathbb Z}_0^{+}}
    \newcommand{\parikh}{\Psi}
    \def\e{\mathrm e}
    \newcommand{\I}{\mathrm i}
    \def\d{\,\mathrm{d}}


    \newtheorem{definition}{Definition}
    \newtheorem{theorem}{Theorem}[section]
    \newtheorem{lemma}[theorem]{Lemma}
    \newtheorem{proposition}[theorem]{Proposition}
    \newtheorem{corollary}[theorem]{Corollary}
    \newtheorem{remark}{Remark}

\title{Balances of $m$-bonacci words}


\author[1]{Karel B\v{r}inda\thanks{karel.brinda{@}fjfi.cvut.cz}}
\author[1]{Edita Pelantov\'{a}\thanks{edita.pelantova{@}fjfi.cvut.cz}}
\author[2]{Ond\v{r}ej Turek\thanks{ondrej.turek{@}kochi-tech.ac.jp}}
\affil[1]{FNSPE Czech Technical University in Prague\\
Trojanova 13, 120 00 Praha 2, Czech Republic}
\affil[2]{Laboratory of Physics\\
Kochi University of Technology\\
Tosa Yamada, Kochi 782-8502, Japan
}

\maketitle

\begin{abstract}
    The $m$-bonacci word is a generalization of the Fibonacci word to
    the $m$-letter alphabet $\A = \{0,\ldots,m-1\}$. It is the unique
    fixed point of the Pisot--type  substitution $ \varphi_m:
    0\to 01,\ 1\to 02,\ \ldots,\ (m-2)\to0(m-1),\text{ and }(m-1)\to0$.
    A result of Adamczewski implies the existence of constants $c^{(m)}$  such that  the $m$-bonacci word is $c^{(m)}$-balanced, i.e.,
    numbers of letter $a$ occurring in two  factors  of the same length
     differ at most by $c^{(m)}$ for any letter $a\in \A$.
    The constants $c^{(m)}$ have been already determined for $m=2$ and
    $m=3$. {In this paper} we study the bounds $c^{(m)}$ for a general
    $m\geq2$. We show that the $m$-bonacci word is  $( \lfloor
    \kappa m   \rfloor +12)$-balanced, where $\kappa  \approx 0.58$.
    For $m\leq 12$, we improve the constant  $c^{(m)}$ by a computer
    numerical calculation   to the value
    $\lceil\frac{m+1}{2}\rceil$.
\end{abstract}



\section{Introduction}

The $m$-bonacci word is a generalization of the Fibonacci word to
 the $m$-letter alphabet $\A = \{0,\ldots,m-1\}$. It is the unique fixed point of the substitution $\varphi = \varphi_m$ given by the prescription
\begin{equation}\label{eq_mbon_subst}
    0\to01,\ 1\to 02,\ \ldots,\ (m-2)\to0(m-1),\text{ and }(m-1)\to0.
\end{equation}
In particular, for $m=3$, we obtain the substitution $0\to01,\ 1\to02,\ 2\to0$ with the fixed point \\
\[
0   1   0   2   0   1   0   0   1   0   2   0   1   0   1   0   2   0   1   0   0   1   0   2   0   1   0   2   0   1   0   0   1   0   2   0   1   0   1   0   2   0   1   0   0   1   0   2   0   1   0   0   1   0   2   0   1   0   1   0   2   0   1   0   0   1   0   2   0   1   0   2   0   1   0   0 \cdots\,,
\]
usually called the Tribonacci word.

The aim of this article is to study a certain combinatorial property of the $m$-bonacci word for a general $m$. Namely, we examine the balance property, which describes
a certain uniformity of occurrences of letters in an infinite word. In order to give its rigorous definition, let us precise the notation we will use in the sequel. A factor of an infinite word $\u = \u_0\u_1\u_2\cdots\in\A^\N$ is any finite string in the form $w=\u_i\u_{i+1}\cdots\u_{i+n-1}$ for certain $i\in\N_0$, $n\in\N$, where $|w|=n$ is the length of the factor $w$. The language of an infinite word $\u$, denoted by $\L(\u)$, is the set of all its factors. The number of occurencies of a given letter $a\in\A$ in a factor $w$ is denoted by $|w|_a$. Clearly, $\sum_{a\in\A}|w|_a = |w|$. The balance property is related to the variability of $|w|_a$ within the meaning of the following definition.

\begin{definition}
    Let $c$ be a positive integer. An infinite word $\u \in \A^\N$ is said to be $c$-balanced if
     \[
        |w|_a - |v|_a \leq c
     \]
     for all factors $w,v \in \L(\u)$ of the same length and for each letter $a\in\A$.
\end{definition}
The notion of a $1$-balanced word (originally referred to as ``balanced word'')
 has been used by  \fixE{Morse and Hedlund } already  in~1940 \cite{Morse2} for a characterization of Sturmian sequences.  Since the Fibonacci word (in our notation $2$-bonacci word)  is Sturmian,  it is $1$-balanced.

It was expected and announced in several papers since 2000 that
the Tribonacci word is $2$-balanced~\cite{CFZ, Berstel, Vuillon}.
This statement has been proved in 2009 (in two different ways) by
Richomme, Saari and Zamoboni \cite{Zamb}. As for a general
$m\geq2$, in 2009 Glen and Justin \cite{GlenJustin}
\fixE{mentioned } ``the $k$-bonacci word
is $(k-1)$-balanced'',
but to the best of our knowledge, no proof of this proposition has ever been published.

The $m$-bonacci words belong to a
broad class called Arnoux--Rauzy words. In the last ten years, balance properties of Arnoux--Rauzy words have been intensively studied. For the most recent
results and a nice overview see \cite{BCS}.

The works of Adamczewski on discrepancy and balance properties of fixed points of primitive substitutions \cite{Adam,Adam2} imply the existence of finite constants $c^{(m)}$ such that the $m$-bonacci word is $c^{(m)}$-balanced. Namely, Adamczewski proved that if all eigenvalues of the matrix of substitution  except the dominant one are of modulus less than 1, then the fixed point of the primitive substitution is $c$-balanced for some $c$.  It is well known (and explicitly shown in our text as well)  that  the substitution defined  by \eqref{eq_mbon_subst}  satisfies the Adamczewski condition.

In the present article, we approach the problem of determining $c^{(m)}$ by refining the matrix  method used by Adamczewski in~\cite{Adam,Adam2} (and also by Richomme, Saari, Zamboni in \cite{Zamb} in their Proof 2). 
Small values of $m$ can be treated numerically. We show that
\begin{itemize}
    \item the $4$-bonacci word and the $5$-bonacci word are $3$-balanced but not $2$-balanced;
    \item for $m=6,7,\ldots,12$ the $m$-bonacci word is $\lceil\frac{m+1}{2}\rceil$-balanced, Theorem \ref{CtyriPet}.
\end{itemize}
The approach works \fixK{for} a general $m$ as well. We prove the following theorem.

\textbf{Theorem.} (Theorem~\ref{hlavni}.) \emph{The $m$-bonacci word is $c^{(m)}$-balanced with
\[
c^{(m)} = \lfloor   \kappa m   \rfloor +12,
\]
where $\kappa = \frac{2}{\pi}\int_{0}^{2\pi}\frac{1-\cos x}{(5- 4 \cos x)\ln (5-4\cos x)} {\rm d}x   \approx 0.58$.}

Our results confirm the bound $c=m-1$ proposed by Glen and Justin
for all $\fixE{m\leq 12}$ and $ m\geq 29
$. Moreover, it turns out that the formerly proposed bound $c=m-1$
is far from being optimal except for a few small values of $m$.

\fixE{ Our article is organized as
follows: Section~\ref{Section 2} explains relationship between balance and
discrepancy and gives  a formula estimating  the balance constant
using spectrum of the matrix $M$ of substitution
\eqref{eq_mbon_subst}. In Section~\ref{Numerics},  we present results obtained
by computer evaluation of this formula. In Section~\ref{relations},  we show
that for estimating the balance constant $c$ we can concentrate on
the letter $0$ only. Sections~\ref{sumace} and \ref{Section 6} are devoted to the proof of
the main theorem.  Our proof requires very detailed information
about spectrum of the matrix $M$; in Appendix we use standard
methods of calculus to describe this spectrum.}


\section{Balance property and discrepancy}\label{Section 2}

This section describes the main idea that will be later applied to find  \fixE{for any letter $a\in\{0,\ldots,m-1\}$
upper 
bound on the letter balance constant
\[
    c_a := \max \{ |w|_a - |v|_a \,: \, v,w \in \L(\u) \hbox{ and }  |w|=|v|
\}\,.
\]
} The derivation of these bounds uses the following two
ingredients.
\begin{itemize}
    \item the $m$-bonacci sequence defined recursively
        \[
            T_0 = T_1 = \ldots = T_{m-2} = 0, \qquad T_{m-1}=1
        \]
        and
        \begin{equation}\label{eq_def_mbon_posl}
            T_n = T_{n-1} + T_{n-2} + \ldots + T_{n-m}
        \end{equation}
        for any $n\geq m$;
    \item
        zeros $\beta \equiv \beta_0>1, \beta_1, \ldots, \beta_{m-1}$ of the polynomial
        \[
            p(x) = x^m - x^{m-1} - \ldots - x - 1.
        \]
\end{itemize}

It is well known that $p(x)$ is an irreducible polynomial, its root $\beta$ belongs to the interval $(1,2)$, and the other roots (conjugates of $\beta$) are all of modulus less than $1$.
From now on\fixK{,} we order the roots $\beta_1,\ldots,\beta_{m-1}$ according to their arguments, i.e.,
\begin{equation}\label{numbering}
\fixE{0\leq}
\arg(\beta_1)\leq\arg(\beta_2)\leq\cdots\leq\arg(\beta_{m-1})\fixE{<
2\pi}\,.
\end{equation}

The $m$-bonacci word is a fixed point of a primitive substitution. Therefore, density $\mu_a$ of any letter $a\in\A$ is well defined and positive, i.e.,
\[
    \mu_a = \lim_{n\to+\infty}\frac{|\u[n]|_a}{n} > 0,
\]
where $\u[n]$ the prefix of $\u$ of length~$n$. We refer to \cite{Queffelec}, where the problem of letter densities is studied in detail.

The value $\mu_a$ can be interpreted in the way that the ``expected'' number of letters $a$ in the prefix $\u[n]$ is $\mu_a n$.
A simple consequence of the definition of $\mu_a$ is the following observation.
\begin{observation}\label{obs_hust}
    For any $\varepsilon>0$ and for any positive integer $N$, there exist factors $v$ and $w$ in $\L(\u)$ such that
    \[
        |v|=|w|=N, \qquad |w|_a \geq \mu_a N - \varepsilon  \qquad \text{and} \qquad  |v|_a \leq \mu_a N + \varepsilon.
    \]
\end{observation}
\begin{proof}
    Assume that there exist $\varepsilon > 0$ and $N \geq 1$ such that for any factor $w$ of length $N$, the inequality $|w|_a < \mu_a N - \varepsilon$ holds. It means that for the prefix of $\u$ of length $n=kN$, we obtain $|\u[n]| = |\u[kN]|_a < (\mu_aN - \varepsilon)k$. This implies $\mu_a = \lim\limits_{n\to+\infty}\frac{\fixK{|}\u[n]\fixK{|}_a}{n} = \lim\limits_{k\to+\infty}\frac{\fixK{|}\u[kN]\fixK{|}_a}{kN} < \mu_a - \frac{\varepsilon}{N}$, which is a contradiction. The proof of existence of $v$ is analogous.
\end{proof}

The difference between the expected and actual number of letters $a$ defines the discrepancy function $D_a: \N \to \R$;
\[
    D_a(n) = |\u[n]|_{\fixK{a}} - \mu_a n
\]
for any $n\in\N$.

\begin{lemma}\label{lem_ca_pomoci_delt}
    For any letter $a$, denote
    \[
        \Delta_a := \sup_{n\in\N} D_a(n) - \inf_{n\in\N} D_a(n).
    \]
    Then $\Delta_a \leq c_a \leq 2 \Delta_a$.
\end{lemma}
\begin{proof}
    Let $w,v \in \L(\u)$ be factors of the same length such that $c_a = |w|_a - |v|_a$. We can find prefixes $W$ and $V$ of $\u$ such that $Ww$ and $Vv$ are prefixes of $\u$ as well. Obviously
    \begin{align*}
        |w|_a - |v|_a &= |Ww|_a - |W|_a - |Vv|_a + |V|_a = D_a(|Ww|) - D_a(|W|) - D_a(|Vv|) + D_a(|V|) \\
        &\leq
        2\sup_{n\in\N} D_a(n) - 2\inf_{n\in\N} D_a(n) = 2\Delta_a.
    \end{align*}
    To deduce the lower bound on $c_a$, let us choose $\varepsilon>0$. There exist prefixes of $\u$, say $\u[n_1]$ and $\u[n_2]$\fixK{,} such that $D_a(n_1) > \sup_{n\in\N} D_a(n) - \varepsilon$ and $D_a(n_2) < \inf_{n\in\N} D_a(n) + \varepsilon$, or equivalently
    \begin{align*}
        |\u[n_1]|_a &> \mu_a n_1 + \sup_{n\in\N} D_a(n) - \varepsilon, \\
        |\u[n_2]|_a &< \mu_a n_2 + \inf_{n\in\N} D_a(n) + \varepsilon.
    \end{align*}

    First suppose that $n_1 > n_2$ and put $N := n_1 - n_2$. Denote the suffix of $\u[n_1]$ of length $N$ by $\tilde{W}$. Then $\tilde{W}$ contains at least $\mu_a N + \sup_{n\in\N} D_a(n) - \inf_{n\in\N} D_a(n) - 2\varepsilon$ letters $a$.

According to Observation~\ref{obs_hust}, there exists a factor $W$ of length $N$ such that $|W|_a \leq \mu_a N + \varepsilon$.
Hence $c_a \geq |\tilde{W}|_a - |W|_a \geq \sup_{n\in\N} D_a(n) - \inf_{n\in\N} D_a(n) - 3\varepsilon =\Delta_a -3\varepsilon.$

The case $n_1 < n_2$ is analogous.
\end{proof}

To find the value $\Delta_a$, we apply the method of Adamczewski used in~\cite{Adam,Adam2}. Let us first recall the notation used in this method.

Let $M$ be a matrix of the substitution~\eqref{eq_mbon_subst}. Since entries of $M$ are defined as $M_{a,b} = |\varphi(b)|_a$ for $a,b\in\{0,1,\ldots,m-1\}$, we have
\[
    M =
    \begin{pmatrix}
        1 & 1 & 1 & \ldots & 1 & 1 \\
        1 & 0 & 0 & \ldots & 0 & 0 \\
        0 & 1 & 0 & \ldots & 0 & 0 \\
        \vdots\\
        0 & 0 & 0 & \ldots & 1 & 0
    \end{pmatrix}
    \in \R^{m\times m}.
\]

By $\parikh(w)$ we denote the Parikh vector of the word $w \in
\A^{*}$, i.e., $\parikh(w) = \left( |w|_0, |w|_1,\ldots,|w|_{m-1}
\right)^{\intercal}$. The matrix of a substitution helps
effectively calculate the Parikh vector of an image $w$
under~$\varphi$. It is easy to see that
\begin{equation}\label{eq_obraz_pres_parikh} 
    \parikh(\varphi(w)) = M \parikh(w) \qquad \text{for any } w\in\A^{*}.
\end{equation}

\begin{lemma}\label{lem_parikh_prefixu}
    For any prefix $\u[n]$ of the $m$-bonacci word $\u$, there \fixE{exist }$\ell \in \N$ and $\delta_0,\delta_1,\ldots,\delta_\ell \in \{0,1\}$ such that
    \begin{equation}\label{eq_parikh_jako_soucet} 
         \parikh(\u[n]) = \sum_{k=0}^{\ell}\delta_k M^k \parikh(0).
    \end{equation}

    Moreover, for any choice \fixK{of} $\ell\in\{0,1,2,\ldots\}$ and $\delta_0,\ldots,\delta_\ell\in\{0,1\}$, there exists a prefix $\u[n]$ of $\u$ such that~\eqref{eq_parikh_jako_soucet} holds.
\end{lemma}
\begin{proof}
    According to result~\cite{Dum}, for any prefix there exist words $E_\ell \not=\epsilon, E_{\ell-1},\ldots,E_1,E_0$ ($\epsilon$ is \fixK{the} empty word) such that
    \begin{equation}\label{eq_prefix_obrazama}
        \u[n] = \varphi^\ell(E_\ell)\varphi^{\ell-1}(E_{\ell-1})\cdots\varphi(E_1)\fixK{E_0}
    \end{equation}
    and for any $k$, the word $E_k$ is a proper prefix of $\varphi(a)$ for some letter $a\in\A$.

    For our substitution $\varphi$, the only proper prefixes of $\varphi(a)$ are $E_k=\epsilon$ and $E_k=0$. Since the Parikh vector of a concatenation of words is the sum of their Parikh vectors, we have
    \[
        \parikh(\u[n]) = \sum_{k=0}^{\ell}\delta_k\parikh\left(\varphi^k(0)\right),
    \]
    where $\delta_k=1$ if $E_k=0$ and $\delta_k=0$ if $E_k = \epsilon$. Applying formula~\eqref{eq_obraz_pres_parikh} to $\parikh(\varphi^k(0))$, we get~\eqref{eq_parikh_jako_soucet}.

    In general, not all sequences of $E_\ell,E_{\ell-1},\ldots,E_1,E_0$ correspond to a prefix of $\u$. The relevant sequences are described by paths in so called prefix graph of substitution. Nevertheless, since for our substitution the equality $\varphi^m(0) = \varphi^{m-1}(0) \varphi^{m-2}(0) \fixK{\cdots} \varphi(0) 0$ holds, any choice of $E_i \in\{\epsilon,0\}$ gives a prefix of $\u$.
\end{proof}

Knowledge of \fixK{the} Parikh vector $\parikh(\u[n])$ enables us to compute discrepancy $D_a(n)$. To make arithmetic manipulation more elegant, Adamczewski denotes row vectors
\begin{align*}
    h^{(0)} &= (1,0,\ldots,0) - \mu_0(1,1,\ldots,1),\\
    h^{(1)} &= (0,1,\ldots,0) - \mu_1(1,1,\ldots,1),\\
    & \vdots \\
    h^{(m-1)} &= (0,\ldots,0,1) - \mu_{m-1}(1,1,\ldots,1),
\end{align*}
and expresses the discrepancy as the scalar product
\begin{equation}\label{eq_diskrepance_parikh} 
    D_a(n) = h^{(a)}\parikh(\u[n]).
\end{equation}
Verification of the formula is straightforward.

Now we can formulate the main tool for estimation of $c_a$.

\begin{proposition}\label{prop_g}
    For any $a\in\{0,1,\ldots,m-1\}$ and $k\in\N$, denote
\begin{equation}\label{g(a,k) def}
g(a,k)=\left|\varphi^k(0)\right|_a-\mu_a\cdot\left|\varphi^k(0)\right|\,,
\end{equation}
where $\mu_a$ is the density of the letter $a$ in $\u$.
Then
    \begin{equation}\label{eq_def_g} 
        g(a,k) = T_{k+m-a-1} - \frac{1}{\beta^{a+1}} T_{k+m}
    \end{equation}
and
    \[
        \sup_{n\in\N} D_a(n) - \inf_{n\in\N} D_a(n) = \sum_{k=0}^{+\infty}|g(a,k)|
    \]
\end{proposition}

\begin{proof}
    At first, since $g(a,k)$ is nothing but $D_a(|\varphi^k(0)|)$, equation~\eqref{eq_diskrepance_parikh} gives $g(a,k)=h^{(a)}\Psi(\varphi^k(0))$.
    Using equation~\eqref{eq_obraz_pres_parikh}, we obtain $\Psi(\varphi^k(0)) = M^k \parikh(0)$, hence
    \begin{equation}\label{predctrnact}
    g(a,k) = h^{(a)}M^k \parikh(0)\,.
    \end{equation}
    This expression combined with equations~\eqref{eq_parikh_jako_soucet} and \eqref{eq_diskrepance_parikh} gives $D_a(n) = \sum_{k=0}^{\ell}\delta_k g(a,k)$, where $\delta_k\in\{0,1\}$.
    Clearly, \
    $\sup\limits_{n\in\N} D_a(n) \leq \!\! \sum\limits_{\begin{subarray}{c}k=0\\g(a,k)>0\end{subarray}}^{+\infty}\!\!g(a,k)$ \  and \
    $\inf\limits_{n\in\N} D_a(n) \geq \!\! \sum\limits_{\begin{subarray}{c}k=0\\g(a,k)<0\end{subarray}}^{+\infty}\!\!g(a,k)$.
    According to Lemma~\ref{lem_parikh_prefixu}, any choice of $\delta_i$'s corresponds to a prefix of $\u[n]$, and, therefore, the equalities are reached in the previous inequalities. To sum up,
    \[
        \sup_{n\in\N} D_a(n) - \inf_{n\in\N} D_a(n) =
         \sum\limits_{\begin{subarray}{c}k=0\\g(a,k)>0\end{subarray}}^{+\infty}\!\!g(a,k) -  \sum\limits_{\begin{subarray}{c}k=0\\g(a,k)<0\end{subarray}}^{+\infty}\!\!g(a,k)=
        \sum_{k=0}^{+\infty}|g(a,k)|.
    \]

    In order to prove equation~\eqref{eq_def_g}, let us observe that
    \[
        \begin{pmatrix}
            T_n\\
            T_{n-1}\\
            \vdots\\
            T_{n-m+1}
        \end{pmatrix} = M
        \begin{pmatrix}
            T_{n-1}\\
            T_{n-2}\\
            \vdots\\
            T_{n-m}
        \end{pmatrix}.
    \]
    Since $\left(T_{m-1},T_{m-2},\ldots,T_0\right) = \left(1,0,0,\ldots,0\right)=\left(\parikh(0) \right)^\intercal,$
    we get using \eqref{predctrnact}
    \begin{equation}\label{ctrnact}   
        g(a,k) = h^{(a)}M^k \parikh(0) = h^{(a)}\left(T_{m+k-1},T_{m+k-2},\ldots,T_k \right)^\intercal\,.
    \end{equation}

    It is readily seen that the vector $\vec{\mu} = \left( \beta^{-1},\beta^{-2},\ldots,\beta^{-m} \right)^\intercal$ is an eigenvector of $M$ corresponding to the dominant eigenvalue $\beta$. Moreover, sum of components of $\vec{\mu}$ equals $1$. It is well known that a vector $\vec{\mu}$ with these properties is the vector of letter densities,  see \cite{Queffelec}.  It means that for any letter $a\in\{0,1,\ldots,m-1\}$, the density of letter $a$ is $\mu_a = \beta^{-1-a}$.  If we apply this fact to~\eqref{ctrnact} and use the relation~\eqref{eq_def_mbon_posl}, we find
    \[
    g(a,k)=T_{m+k-a-1} - \beta^{-a-1} T_{m+k}\,.
    \]
\end{proof}

\begin{corollary}\label{bound}
The balance constants of the $m$-bonacci word satisfy
\begin{equation}\label{c_a}
c_a \leq 2 \sum_{k=0}^{+\infty} |g(a,k)|
\end{equation}
for all $a\in\A$.
\end{corollary}

\begin{proof}
The estimate follows easily from Lemma~\ref{lem_ca_pomoci_delt} and Proposition~\ref{prop_g};
\[
c_a \leq 2 \Delta_\fixK{a}  = 2\left(\sup_{n\in\N} D_\fixK{a}(n) - \inf_{n\in\N} D_\fixK{a}(n)\right) = 2\sum_{k=0}^{+\infty}|g(\fixK{a},k)|\,.
\]
\end{proof}

\begin{remark}
    To estimate the sum $\sum_{k=0}^{+\infty}|g(a,k)|$, we will use the explicit formula for elements $T_n$ of the $m$-bonacci sequence. The characteristic equation of~\eqref{eq_def_g} is the polynomial $p(x)$ with zeros $\beta= \beta_0,\beta_1,\ldots,\beta_{m-1}$. Hence there exist constants $a_0,a_1,\ldots,a_{m-1}\in\C$ such that
    \[
        T_n = a_0\beta_0^n + a_1\beta_1^n+\ldots+a_{m-1}\beta_{m-1}^{n}.
    \]
    The constants $a_0,a_1,\ldots,a_{m-1}$ depend on the initial values $T_0,T_1,\ldots,T_{m-1}$ only. A standard calculation provides $T_n=\sum_{j=0}^{m-1}\frac{1}{p'(\beta_j)}\beta_j^n$, where $p'$ denotes the derivative of the characteristic polynomial $p$.

    Using~\eqref{eq_def_g}, we can conclude with
    \begin{equation}\label{eq_g_pomoci_char_pol} 
        g(a,k) = \sum_{j=1}^{m-1} \left( {\frac{1}{\beta_j^{a+1}} - \frac{1}{\beta^{a+1}}} \right)\frac{1}{p'(\beta_j)}\beta_j^{k+m}.
    \end{equation}
\end{remark}


\section{Numerical upper bounds on balance constant}\label{Numerics}

According to Corollary~\ref{bound}, the
\fixE{{letter balance constants}} of the
$m$-bonacci word $\u$ can be estimated by the formula
\[
    c_a \leq \left\lfloor 2 \sum_{k=0}^{+\infty} |g(a,k)| \right\rfloor
\]
for any letter $a\in\{0,1,\ldots,m-1\}$ and for all $m\geq2$.

In this section we estimate the expressions $\left\lfloor 2 \sum_{k=0}^{+\infty} |g(a,k)| \right\rfloor$ using a computer calculation. The calculations are very time-consuming for $m$ above $10$, therefore, we confine ourselves to $m\leq12$.

The calculation is based on the following strategy. We sum up the first $n$ members of $(|g_{(a,k)}|)_{k=0}^{+\infty}$ and estimate the rest of them;
\[
    \sum_{k=0}^{+\infty}|g{(a,k)}| \leq \sum_{k=0}^{n - 1}|g{(a,k)}| + E, \qquad\text{ where } E \text{ satisfies } \qquad E \geq \sum_{k=n}^{+\infty}|g{(a,k)}|.
\]
Formula~\eqref{eq_g_pomoci_char_pol} provides setting
\[
    E_{a,n} := |\beta_2|^{n} \sum_{j=1}^{m-1}\left|\left(\frac{1}{\beta_j^{a+1}} - \frac{1}{\beta^{a+1}}\right)\frac{1}{p'(\beta_j)}\right|\frac{|\beta_j|^{n}}{1 - |\beta_j|}.
\]
To conclude, we have to find an $n$ big enough to satisfy
\begin{equation}\label{eq_rovne_B}
    \left\lfloor 2 \sum_{k=0}^{n-1}\left|g{(a,k)}\right| \right\rfloor =
    \left\lfloor 2\left(\sum_{k=0}^{n-1}\left|g{(a,k)}\right| + E_{a,n} \right) \right\rfloor.
\end{equation}

Since we always compute on machines working in a finite precision, it is desirable to reduce the work with non-integer numbers. Therefore, we make use of the fact that, for a fixed letter $a$ and the alphabet cardinality $m$, the sequence of numbers $g{(a,k)}$ satisfies the $m$-bonacci recurrence relation
\[
g{(a,n+m)} = g{(a,n+m-1)} + \ldots + g{(a,n)}\,,
\]
which follows from Proposition~\ref{prop_g}.

Let us demonstrate the method on the $4$-bonacci word. The first step is calculating\footnote{The calculation must be performed in an environment working in enough precision, e.g., Wolfram Mathematica.} $\mathrm{sgn}\, g{(a,k)}$ from~\eqref{eq_def_g} for all $k \in \{0,\ldots,m-1\}$ (illustrated in Table~\ref{tab_table1_4bonacci}). Then we express $\sum_{k=0}^{n- 1}|g{(a,k)}|$ as an integer combination (IC) of $\left(\begin{smallmatrix}g{(a,0)}\\ \vdots \\ g{(a,m-1)}\end{smallmatrix}\right)$, which can be rewritten in the form $p + \frac{q}{\beta^{a+1}}$ for some $p,q \in \Z$ (this follows from Proposition~\ref{prop_g}) and then evaluated\footnotemark[\value{footnote}] \fixK{(}see Table~\ref{tab_table2_4bonacci}). The final step is verification of the equality~\eqref{eq_rovne_B}.

To make our procedure reliable with respect to possible rounding errors, we replace the estimated error $E_{a,m}$ by a constant $E > E_{a,m}$. If~\eqref{eq_rovne_B} holds, it is equal to the desired upper bound of $c_a$ (but it may not be optimal). In the opposite case, we must increase~$n$ and repeat the procedure.

Our results obtained for $m\in\{2,\ldots,12\}$ are summarized in Table~\ref{tab_table3}.

%
%
\begin{table}
\centering
\caption{{\bf $4$-bonacci} -- $g(a,k)$  with quadruple of  integer coefficients in linear  combination  of $g(a,0),\ldots,g(a,3)$ and its signum.}
\label{tab_table1_4bonacci}
\begin{tabular}{c|c|cccc}
& IC of $(g(a,k))_{k=0}^{3}$& $a=0$& $a=1$& $a=2$& $a=3$\\\hline
$g(a,0)$ & $(1,0,0,0)$  & $+$ & $-$ & $-$ & $-$ \\
$g(a,1)$ & $(0,1,0,0)$  & $-$ & $+$ & $-$ & $-$ \\
$g(a,2)$ & $(0,0,1,0)$  & $-$ & $-$ & $+$ & $-$ \\
$g(a,3)$ & $(0,0,0,1)$  & $-$ & $-$ & $-$ & $+$ \\
$g(a,4)$ & $(1,1,1,1)$  & $+$ & $-$ & $-$ & $-$ \\
$g(a,5)$ & $(1,2,2,2)$  & $-$ & $+$ & $-$ & $-$ \\
$g(a,6)$ & $(2,3,4,4)$  & $-$ & $-$ & $+$ & $-$ \\
$g(a,7)$ & $(4,6,7,8)$  & $-$ & $-$ & $-$ & $+$ \\
$g(a,8)$ & $(8,12,14,15)$  & $+$ & $+$ & $-$ & $-$ \\
$g(a,9)$ & $(15,23,27,29)$  & $-$ & $+$ & $+$ & $-$ \\
$g(a,10)$ & $(29,44,52,56)$  & $-$ & $-$ & $+$ & $+$ \\
$g(a,11)$ & $(56,85,100,108)$  & $+$ & $-$ & $-$ & $+$ \\
$g(a,12)$ & $(108,164,193,208)$  & $+$ & $+$ & $-$ & $-$ \\
\end{tabular}
\end{table}

%
%
\begin{table}
\caption{{\bf $4$-bonacci} -- Estimates of $\sum_{k=0}^{+\infty}|g(a,k)|$ and the resulting upper bound on $c_a$.}
\label{tab_table2_4bonacci}
\centering
\begin{tabular}{p{2.5cm}||p{2.5cm}|p{2.5cm}|p{2.5cm}|p{2.5cm}}
 & $a=0$ & $a=1$ & $a=2$ & $a=3$\\\hline\hline
$\displaystyle\sum_{k=0}^{12}|g(a,k)|$ as IC & $\left(\begin{smallmatrix}123\\183\\215\\232\end{smallmatrix}\right)$ & $\left(\begin{smallmatrix}39\\63\\71\\76\end{smallmatrix}\right)$ & $\left(\begin{smallmatrix}-133\\-201\\-233\\-254\end{smallmatrix}\right)$ & $\left(\begin{smallmatrix}-47\\-71\\-83\\-86\end{smallmatrix}\right)$\\\hline
$\displaystyle\sum_{k=0}^{12}|g(a,k)|$ symbolic & $1664-\frac{3205}{\beta }$ & $286-\frac{1057}{\beta ^2}$ & $\frac{3499}{\beta ^3}-487$ & $\frac{1209}{\beta ^4}-86$\\\hline
$\displaystyle\sum_{k=0}^{12}|g(a,k)|$ numerical & $1.2778$ & $1.5157$ & $1.5611$ & $1.5776$\\\hline
$E_{a,13}$ & $0.20054$ & $0.22213$ & $0.25916$ & $0.31056$\\\hline
$\displaystyle\sum_{k=0}^{12}|g(a,k)| + E$ & $1.49844$ & $1.76006$ & $1.84618$ & $1.91919$\\\hline\hline
$c_a$ upper bound & $2$  & $3$  & $3$  & $3$ \\
\end{tabular}
\end{table}

%
%
\begin{table}
\centering
\caption{Upper estimates of $c_a$ for $m\in\{2,\ldots,12\}$, $a \in \{0,\ldots,m-1\}$.}
\label{tab_table3}
\begin{tabular}{l||c|c|c|c|c|c|c|c|c|c|c|c}
$m$ \ $\backslash$ \ $a$ & $0$ & $1$ & $2$ & $3$ & $4$ & $5$ & $6$ & $7$ & $8$ & $9$ & $10$ & $11$  \\\hline\hline
$2$  & $1$ & $1$ &  $\times$ & $\times$ & $\times$ & $\times$ & $\times$ & $\times$ & $\times$ & $\times$ & $\times$ & $\times$  \\\hline
$3$  & $2$ & $2$ & $2$ & $\times$ & $\times$ & $\times$ & $\times$ & $\times$ & $\times$ & $\times$ & $\times$ & $\times$ \\\hline
$4$  & $2$ & $3$ & $3$ & $3$ & $\times$ & $\times$ & $\times$ & $\times$ & $\times$ & $\times$ & $\times$ & $\times$\\\hline
$5$  & $2$ & $3$ & $3$ & $3$ & $3$ & $\times$ & $\times$ & $\times$ & $\times$ & $\times$ & $\times$ & $\times$\\\hline
$6$  & $3$ & $3$ & $4$ & $4$ & $4$ & $4$ & $\times$ & $\times$ & $\times$ & $\times$ & $\times$ & $\times$\\\hline
$7$  & $3$ & $4$ & $4$ & $4$ & $4$ & $4$ & $4$ & $\times$ & $\times$ & $\times$ & $\times$ & $\times$\\\hline
$8$  & $3$ & $4$ & $4$ & $4$ & $4$ & $4$ & $4$ & $4$ & $\times$ & $\times$ & $\times$ & $\times$\\\hline
$9$  & $3$ & $4$ & $5$ & $5$ & $5$ & $5$ & $5$ & $5$ & $5$ & $\times$ & $\times$ & $\times$\\\hline
$10$  & $3$ & $5$ & $5$ & $5$ & $5$ & $5$ & $5$ & $5$ & $5$ & $5$ & $\times$ & $\times$\\\hline
$11$  & $4$ & $5$ & $5$ & $6$ & $6$ & $6$ & $6$ & $6$ & $6$ & $6$ & $6$ & $\times$ \\\hline
$12$  & $4$ & $5$ & $6$ & $6$ & $6$ & $6$ & $6$ & $6$ & $6$ & $6$ & $6$ & $6$ \\\hline
\end{tabular}
\end{table}

To find lower bounds on the constant $c$, one needs to find two factors $v, w$ of the $m$-bonacci word that are of the same length with $|w|_a - |v|_a$ big enough. Computer searching in the set of all factors is very time-consuming. Nevertheless, for any given $m\geq4$ and any $a \in \{1,\ldots,m-1\}$, a modification of the abelian co-decomposition method~\cite{Tur12} allowed us to find a pair of factors $v, w$ of the $m$-bonacci word such that $|v|=|w|$ and $|v|_a - |w|_a = 3$.
For instance,
if $m=4$, the words
\[
v=1\varphi^{12}(0)\varphi^{9}(0)\varphi^{5}(0)\varphi^{2}(0)\,,
\]
\[
w=\left(\varphi^{9}(0)\varphi^{8}(0)\varphi^{5}(0)\varphi^{2}(0)\right)^{-1}\varphi^{11}(00)\varphi^{10}(0)\varphi^{7}(0)\varphi^{6}(0)\varphi^{4}(0)\varphi^{3}(0)\varphi^{2}(0)0
\]
are factors of $\u$ such that $|v|=|w|=3305$, $|v|_1-|w|_1=3$.
Similarly, if $m=5$, the words
\[
v=1\varphi^{14}(0)\varphi^{11}(0)\varphi^{6}(0)\varphi^{2}(0)\,,
\]
\[
w=\left(\varphi^{11}(0)\varphi^{10}(0)\varphi^{6}(0)\varphi^{2}(0)\right)^{-1}\varphi^{13}(00)\varphi^{12}(0)\varphi^{9}(0)\varphi^{8}(0)\varphi^{7}(0)\varphi^{5}(0)\varphi^{3}(0)\varphi^{2}(0)0
\]
are factors of $\u$ such that $|v|=|w|=15481$, $|v|_1-|w|_1=3$.

Therefore, we can conclude with the following theorem.

\begin{theorem}\label{CtyriPet}
    For $m \in\{4,5\}$, the $m$-bonacci word is $c$-balanced with $c=3$ and this bound cannot be improved.

    For $m \in \{6,\ldots,12\}$, the $m$-bonacci word is $c$-balanced for $c = \lceil \frac{m+1}{2} \rceil$.
\end{theorem}


\section{Balance property of letters in the $m$-bonacci word}\label{relations}

The numerical calculation, performed in Section~\ref{Numerics}, is convenient only for small values of $m$. In the rest of the paper we develop a technique to estimate the constant $c$ for the balance property of the $m$-bonacci word for a general $m$.
The calculation will be again based on formula~\eqref{c_a}, but this time we bring in an improvement. Instead of estimating the sums $\sum_{k=0}^{+\infty}|g(a,k)|$ for all letters $a\in\A$, we show that in case of the $m$-bonacci word, the balance constants $c_a$ for $a = 1,2,\ldots,m-1$ can be estimated by a simple formula in terms of $c_0$ providing that $c_0$ is small enough, see the following observation.

\begin{proposition}\label{prop_c0} 
    Let $m\geq4$. If $c_0\leq2^{m-1}-3$, then
    \begin{equation}\label{eq_prop}
        c_j \leq \left(2-\frac{1}{2^j}\right)c_0 + 4\left(1 - \frac{1}{2^j}\right)
    \end{equation}
    for each $j=1,2,\ldots,m-1$. In particular, the $m$-bonacci word is $c$-balanced with $c=2c_0+3$.
\end{proposition}
With regard to this proposition, it will be sufficient to estimate $\sum_{k=0}^{+\infty}|g(a,k)|$ and use formula~\eqref{c_a} just once, for $a=0$. All the remaining constants $c_a$ for $a=1,\ldots,m-1$ can be then easily estimated using formula~\eqref{eq_prop}.

Before we prove Proposition \ref{prop_c0}, we derive two simple
observations.

\begin{observation}\label{obs_vyjadreni_f0_f_rovno}
    For any factor $f$ of $\u$ and for each $j\in \{1,\ldots,m-1\}$, it holds
    \[
        |f|_0 = |\varphi^j(f)|_j \qquad \text{ and } \qquad |f| = |\varphi^j(f)|_{j-1}.
    \]
\end{observation}

\begin{proof}
    From the form of the substitution~\eqref{eq_mbon_subst}, we see $|w|_{j-1} = |\varphi(w)|_j$ and $|w|=|\varphi(w)|_0$ for any factor $w$ and letter $j=1,2,\ldots,m-1$. Applying these relations on $w=f,\ w=\varphi(f), \ldots,\ w=\varphi^{j-1}(f)$, we get the formulae in the observation.
\end{proof}

\begin{observation}\label{obs_f0_shora}
    If $f$ is a factor of $\u$ such that $|f|\leq2^m$, then $|f|_0\leq\frac{1}{2}|f|+1$.
\end{observation}

\begin{proof}
    The form of the substitution $\varphi$ implies that $00$ is the longest block of zeros occurring in $\u$. Further, with exception of this block, the letter $0$ is always sandwiched by nonzero letters. It is easy to see that the shortest factor $w\not=00$, with the prefix $00$ and the suffix $00$ such that $w$ has no other occurrences of $00$, is the factor $w=0\varphi^m(0)0$. Since $|w|=2^m+1$, any factor f with $|f|\leq2^m$ contains at most one block $00$. This implies the inequality for $|f|_0$ stated in the observation.
\end{proof}

The following lemma is the combinatorial core for the proof of Proposition~\ref{prop_c0}.

\begin{lemma}\label{lemma cj pomoci cj-1}
Let $j\in\{1,\ldots,m-1\}$. If $c_{j-1} \leq 2^m-2$, then
\begin{equation}\label{cj pomoci cj-1}
c_j\leq c_0+2+\frac{c_{j-1}}{2}\,.
\end{equation}
\end{lemma}

\begin{proof}
    With respect to the definition of $c_j$, there exists a pair of factors $v$ and $w$ such that
    \begin{equation}\label{eq_cbalancovanost}  
        |v| = |w| \qquad \text{ and } \qquad |v|_j - |w|_j = c_j.
    \end{equation}

    Without loss of generality, we can assume that $v$ and $w$ is the shortest possible pair satisfying~\eqref{eq_cbalancovanost}. Then $v$ and $w$ are in the form $v=j\cdots j$ and $w=\ell\cdots \ell'$ for certain $\ell,\ell'\not=j$. Moreover, we can assume that $jw$ is a factor of $\u$ (otherwise we replace $w=\u_i\cdots\u_{i+|w|-1}$ by $w'=\u_{i-i'}\cdots\u_{i+|w|-1-i'}$ without violating equations~\eqref{eq_cbalancovanost}).

    Because of the form of $v$, there exists a factor $V=0V'\in\L(\u)$ such that $v=j\varphi^j(V')$. Clearly, $v$ is a suffix of $\varphi^j(0V')=\varphi^j(V)$.

    Let $wzj$ be a factor of $\u$ such that $|z|_j=0$ (we extend the factor $w$ to the right up to the next letter $j$). As $jwzj\in\L(\u)$ by assumption, there exists a factor $W$ such that $wzj=\varphi^j(W0)$.

Observation~\ref{obs_vyjadreni_f0_f_rovno} implies
\begin{itemize}
    \item $|V|_0 = 1+|V'|_0 = 1+|\varphi^j(V')|_j = |j\varphi^j(V')|_j = |v|_j$
    \item $|W|_0 = |W0|_0 - 1 = |\varphi^j(W0)|_j-1 = |wzj|_j-1 = |w|_j$
    \item $|V| = 1 + |V'| = 1 + |\varphi^j(V')|_{j-1} = 1 + |v|_{j-1}$
    \item $|W| = |W0|-1 = |\varphi^j(W0)|_{j-1}-1 = |wzj|_{j-1} - 1$
\end{itemize}

Together, we have deduced
\begin{equation}\label{eq_rozdil_nul_vw} 
    |V|_0 - |W|_0 = c_j \qquad \text{ and } \qquad |V| - |W| = |v|_{j-1} - |w|_{j-1} - |z|_{j-1} + 2.
\end{equation}

We distinguish two cases:
\begin{itemize}
    \item \emph{Case $|V|\leq|W|$.} Let $\hat{V} = Vx$ be a factor of $\u$ such that $|\hat{V}| = |W|$. From the definition of $c_0$ and~\eqref{eq_rozdil_nul_vw} we get $c_0 \geq |\hat{V}|_0 - |W|_0 \geq |V|_0 - |W|_0 = c_j$. Thus $c_j\leq c_0+2+\frac{c_{j-1}}{2}$ holds trivially.
    \item \emph{Case $|V|>|W|$.} Let $\hat{W} = Wy$ be a factor of $\u$ such that $|\hat{W}| = |V|$. Then $c_0 \geq |V|_0 - |\hat{W}|_0 = |V|_0 - |W|_0 - |y|_0 = c_j - |y|_0$ due to~\eqref{eq_rozdil_nul_vw}. To bound length of~$y$, we apply Equation~\eqref{eq_rozdil_nul_vw}. It gives $|y| = |V|-|W| = |v|_{j-1} - |w|_{j-1} - |z|_{j-1} + 2 \leq |v|_{j-1} - |w|_{j-1} + 2 \leq c_{j-1} + 2$. With regard to the assumption $c_{j-1}\leq2^m-2$, we have $|y|\leq2^m$. Therefore, $|y|_0\leq\frac{1}{2}|y|+1$ due to Observation~\ref{obs_f0_shora}. To sum up, $c_0 \geq c_j - \left(\frac{1}{2}(c_{j-1} + 2)+1\right)\geq c_j -2 -\frac{1}{2}c_{j-1}$.
\end{itemize}
\end{proof}

\begin{proof}[Proof of Proposition~\ref{prop_c0}]
    Let us assume that $c_0 \leq 2^{m-1} - 3$. We prove equation~\eqref{eq_prop} by induction on $j$.

    \emph{I. Let $j=1$.} It holds $c_0 \leq 2^{m-1} - 3 \leq 2^m - 2$ by assumption, therefore, we can use Lemma~\eqref{lemma cj pomoci cj-1}. It implies $c_1\leq c_0+2+\frac{c_0}{2}$, hence indeed $c_j \leq \left(2 - \frac{1}{2^1}\right)c_0 + 4\left(1 - \frac{1}{2^1}\right)$.

    \emph{II. Let $j>1$ and equation~\eqref{eq_prop} hold for $j-1$.}
    Inequality $c_0 \leq 2^{m-1} - 3$ implies
    \[
    c_{j-1} \leq \left(2 - \frac{1}{2^{j-1}}\right)c_0 + 4\left(1 - \frac{1}{2^{j-1}}\right) < 2c_0 + 4 \leq 2(2^{m-1} - 3) + 4 = 2^m - 2\,. 
    \]
    It allows us to apply Lemma~\eqref{lemma cj pomoci cj-1}. Equation~\eqref{cj pomoci cj-1} gives
    \[
    c_j \leq c_0 + 2 + \frac{1}{2}c_{j-1} \leq c_0 + 2 + \frac{1}{2}\left( \left( 2 - \frac{1}{2^{j-1}} \right)c_0 + 4\left(1 - \frac{1}{2^{j-1}} \right) \right) = \left(2 - \frac{1}{2^j}\right)c_0 + 4\left(1 - \frac{1}{2^j}\right)\,.
    \]
In particular, \eqref{eq_prop} yields $c_j < 2c_0 + 4$. As $c = \max\{c_j\  : \  j=0,1,\ldots,m-1\}$ and $c$ and $c_0$ are  integers, necessarily $c \leq 2c_0 + 3$.
\end{proof}


\section{Estimate of $\sum_{k=0}^{+\infty}|g(0,k)|$}\label{sumace}

As anticipated in Section~\ref{relations}, the balance constant $c_0$ will be
 obtained using formula~\ref{c_a}. Therefore, we need to estimate
  the sum $\sum_{k=0}^{+\infty}|g(0,k)|$. This is the topic of this section;
   since we deal with the letter $a=0$ only, we abbreviate the symbol $g(0,k)$ to $g(k)$.

The sum $\sum_{k=0}^{+\infty}|g(0,k)|$ will be estimated by splitting it into
two parts, $\sum_{k=0}^{2m-1}|g(k)|$ and $\sum_{k=2m}^{+\infty}|g(k)|$,
and estimating each of them separately.
In Sections \ref{prvni cast} and \ref{druha cast} we show that
\[
\sum_{k=0}^{2m-1}|g(k)|<\frac{5}{4}\qquad \text{and} \qquad
\sum_{k=2m}^{+\infty}|g(k)|<\frac{A}{2\pi}m+1\quad\text{for all $m\geq4$}\,.
\]
\fixE{To get these estimates we will
exploit bounds on absolute values and arguments of zeros of polynomials
$p(x)$, derived in Appendix~\ref{Apendix A}.}

\subsection{An upper bound on the sum   $\sum\limits_{k=0}^{2m-1}|g(k)|$ } \phantom{.} \label{prvni cast}

 At first we express  $g(k)$'s for all $k=0,1,\ldots,2m-1$ and determine their signs.  Recall that $\mu_0=1/\beta$, therefore, due to equation~\eqref{g(a,k) def}, it holds
\begin{equation}\label{g(0,k) def}
g(k)=\left|\varphi^k(0)\right|_0-\frac{1}{\beta}\cdot\left|\varphi^k(0)\right|\,.
\end{equation}
In the sequel we use the following  formula to calculate $g(k)$ for all $k\leq 2m-1$.

\begin{proposition}\label{prop. Fk}
It holds
\begin{equation}\label{Fk}
|\varphi^k(0)|=\left\{\begin{array}{ll}
2^k & \text{for $k=0,\ldots,m-1$}\,; \\
2^k-2^{k-m}-(k-m)2^{k-m-1} & \text{for $k=m,\ldots,2m-1$}\,.
\end{array}\right.
\end{equation}
\end{proposition}

\begin{proof}
The identity $\varphi^k(0)=\varphi\left(\varphi^{k-1}(0)\right)$ together with the substitution~\eqref{eq_mbon_subst} implies
\begin{equation}\label{phi m-1}
|\varphi^k(0)|=2|\varphi^{k-1}(0)|-|\varphi^{k-1}(0)|_{m-1}\,.
\end{equation}
Let us distinguish two cases.

$\bullet$ \emph{Case $k\leq m-1$.}
It holds $\varphi^0(0)=0$ and $|\varphi^{k}(0)|_{m-1}=0$ for all $k\leq m-2$, hence $|\varphi^k(0)|=2|\varphi^{k-1}(0)|$ for all $k\leq m-1$.

$\bullet$ \emph{Case $k\geq m$.}
We prove equation~\eqref{Fk} for $k\in\{m,m+1,\ldots,2m-1\}$ by induction on $k$.

I. $k=m$. We have $|\varphi^{m-1}(0)|_{m-1}=1$, hence $|\varphi^m(0)|=2|\varphi^{m-1}(0)|-1=2^m-1$. Since $2^m-1=2^m-2^{m-m}-(m-m)2^{m-m-1}$, the statement holds true for $k=m$.

II. $k\geq m+1$. Let $|\varphi^{k-1}(0)|=2^{k-1}-2^{k-1-m}-(k-1-m)2^{k-1-m-1}$. The identity $|\varphi^{k-1}(0)|_{m-1}=|\varphi^{k-1-m}(0)|$, valid for every $k\geq m+1$, allows us to use the formula~\eqref{phi m-1} in the form
\[
|\varphi^k(0)|=2|\varphi^{k-1}(0)|-|\varphi^{k-1-m}(0)|\,.
\]
Since $k-1-m<m-1$, we can apply the results obtained above $k\leq m-1$, whence we get
\[
|\varphi^k(0)|=2\left(2^{k-1}-2^{k-1-m}-(k-1-m)2^{k-1-m-1}\right)-2^{k-1-m}=2^k-2^{k-m}-(k-m)2^{k-m-1}\,.
\]
\end{proof}

To determine signs
  of $g(k)$'s def\fixK{i}ned by \eqref{g(0,k) def}, we need a fine estimate on $\beta$ . Let us recall that $\beta$ is the dominant eigenvalue of the matrix  of substitution $M$  and thus a zero of  its characteristic  polynomial  $p(x)=x^m-x^{m-1}-x^{m-2}-\cdots-x-1$ .

\begin{proposition}\label{Prop. g < 2m}
It holds
\begin{align*}
&g(0)=1-\frac{1}{\beta}>0\,; \\
&g(k)=2^{k-1}\left(1-\frac{2}{\beta}\right)<0 \quad \text{for $k=1,\ldots,m-1$}\,; \\
&g(m)=2^{m-1}\left(1-\frac{2}{\beta}\right)+\frac{1}{\beta}>0\,; \\
&g(k)=\left(1-\frac{2}{\beta}\right)\left(2^{k-1}-(k+1-m)2^{k-m-2}\right)+\frac{1}{\beta}\cdot2^{k-m-1}<0 \quad \text{for $k=m+1,\ldots,2m-1$}\,.
\end{align*}
\end{proposition}

\begin{proof}
The formula for $g(0)$ follows immediately from equation~\eqref{g(0,k) def}.

For every $k\geq1$, it holds $|\varphi^k(0)|_0=|\varphi^{k-1}(0)|$, hence
\[
g(k)=|\varphi^{k-1}(0)|-\frac{1}{\beta}\cdot|\varphi^k(0)|\,\fixK{,}
\]
cf. equation~\eqref{g(0,k) def}. All the \fixK{formulae} for $g(k)$ listed in Proposition~\ref{Prop. g < 2m} then follow easily from equation~\eqref{Fk}.

In the rest of the proof we show that $g(0)>0$, $g(m)>0$, and $g(k)<0$ for all $k\in\{1,\ldots,m-1\}\cup\{m+1,\ldots,2m-1\}$.

At first, $\beta\in(1,2)$ immediately implies $g(0)>0$ and $g(k)<0$ for all $k\in\{1,\ldots,m-1\}$.

As for $k=m$, we shall show that
\[
2^{m-1}\left(1-\frac{2}{\beta}\right)+\frac{1}{\beta}>0\,.
\]
This inequality is equivalent to
\[
2-\beta<\frac{1}{2^{m-1}}\,,
\]
which is valid due to~\eqref{x0}
\fixE{from Appendix}, because
$1/2^{m-1}>1/(2^m-(m+1)/2)$ for all $m\geq2$. Similarly, if $k\geq
m+1$, we need to prove that
\[
\left(1-\frac{2}{\beta}\right)\left(2^{k-1}-(k+1-m)2^{k-m-2}\right)+\frac{1}{\beta}\cdot2^{k-m-1}<0 \quad\text{for $k=m+1,\ldots,2m-1$}\,;
\]
i.e.,
\[
2-\beta>\frac{1}{2^{m}-\frac{k+1-m}{2}} \qquad\text{for all $k=m+1,\ldots,2m-1$}\,.
\]
Since $k+1-m\leq m$, the validity immediately follows from inequalities~\eqref{x0}.
\end{proof}

\begin{proposition}\label{zacatek}
It holds
\begin{equation}\label{|g| 0..2m-1}
\sum_{k=0}^{2m-1}|g(k)|=1+\left(\frac{2}{\beta}-1\right)\left[2^m\left(2^{m-1}-1\right)-(m-1)2^{m-2}\right]-\frac{1}{\beta}\left(2^{m-1}-1\right)<1+\frac{1}{4}\,.
\end{equation}
\end{proposition}

\begin{proof}
Proposition~\ref{Prop. g < 2m} implies
\[
\sum_{k=0}^{2m-1}|g(k)|=g(0)-\sum_{k=1}^{m-1}g(k)+g(m)-\sum_{k=m+1}^{2m-1}g(k)\,.
\]
When we substitute for $g(k)$ from Proposition~\ref{Prop. g < 2m}, we obtain
\[
g(0)+g(m)=1+2^{m-1}\left(1-\frac{2}{\beta}\right)\,,
\]
\[
-\sum_{k=1}^{m-1}g(k)=-\sum_{k=1}^{m-1}2^{k-1}\left(1-\frac{2}{\beta}\right)=-(2^{m-1}-1)\left(1-\frac{2}{\beta}\right)\,,
\]
and, in a similar way, we get
\[
-\sum_{k=m+1}^{2m-1}g(k)=-\left(1-\frac{2}{\beta}\right)\left[2^m(2^{m-1}-1)-(m-1)2^{m-2}\right]-\frac{1}{\beta}\left(2^{m-1}-1\right)\,.
\]
Summing up these expressions, we get formula~\eqref{|g| 0..2m-1}.

In the rest of the proof we show that $\sum_{k=0}^{2m-1}|g(k)|<1+1/4$, which is obviously equivalent to
\[
\left(2-\beta\right)\left[2^m\left(2^{m-1}-1\right)-(m-1)2^{m-2}\right]-2^{m-1}+1<\frac{\beta}{4}\,,
\]
and also to
\[
\left(2-\beta\right)\left[2^m\left(2^{m-1}-1\right)-(m-1)2^{m-2}+\frac{1}{4}\right]-2^{m-1}+1<\frac{1}{2}\,.
\]
Using inequality~\eqref{x0}, we obtain
\begin{multline*}
\left(2-\beta\right)\left[2^m\left(2^{m-1}-1\right)-(m-1)2^{m-2}+\frac{1}{4}\right]-2^{m-1}+1 \\
\leq
\frac{1}{2^m-\frac{m+1}{2}}\left[2^m\left(2^{m-1}-1\right)-(m-1)2^{m-2}+\frac{1}{4}\right]-2^{m-1}+1 \\
=
\frac{2^{m-1}-1-\frac{m-1}{4}+\frac{1}{2^{m+2}}}{1-\frac{m+1}{2^{m+1}}}-2^{m-1}+1=\frac{-\frac{m-1}{4}+\frac{1}{2^{m+2}}+\frac{m+1}{4}-\frac{m+1}{2^{m+1}}}{1-\frac{m+1}{2^{m+1}}}=
\frac{1}{2}\cdot\frac{1-\frac{2m+1}{2^{m+1}}}{1-\frac{m+1}{2^{m+1}}}<\frac{1}{2}\,.
\end{multline*}
\end{proof}

\subsection{An upper bound on the sum $\sum\limits_{k=2m}^{+\infty}|g(k)|$ }\label{druha cast}

\begin{proposition}  For any $k \in \mathbb{N}$ we have
\begin{equation}\label{|jenom g(0,k)|}
|g(k)|\leq
\frac{1}{2(m-1)}\sum_{j=1}^{m-1}|\Re(\beta_j)-1|\cdot|\beta_j|^k\,.
\end{equation}
\end{proposition}
\begin{proof}
With regard to equation~\eqref{p}
\fixE{from Appendix},
\[
p'(x)=\frac{(m+1)x^m-2mx^{m-1}}{x-1}-\frac{x^{m+1}-2x^m+1}{(x-1)^2}=\frac{(m+1)x^m-2mx^{m-1}}{x-1}-\frac{p(x)}{x-1}\,.
\]
Since $p(\beta_j)=0$ for every eigenvalue of $M$, we have
\[
p'(\beta_j)=\frac{(m+1)\beta_j^m-2m\beta_j^{m-1}}{\beta_j-1}=\frac{(m+1)\beta_j-2m}{\beta_j-1}\beta_j^{m-1}\,.
\]
Therefore, due to~\eqref{eq_g_pomoci_char_pol},
\[
g(k)=\sum_{j=1}^{m-1}\left(\frac{1}{\beta_j}-\frac{1}{\beta}\right)\frac{\beta_j^{k+m}}{\frac{(m+1)\beta_j-2m}{\beta_j-1}\beta_j^{m-1}}=
\sum_{j=1}^{m-1}\frac{\beta-\beta_j}{\beta}\cdot\frac{\beta_j-1}{(m+1)\beta_j-2m}{\beta_j^{k}}\,.
\]
As $g(k)$ is real, we can write
\begin{equation}\label{g(0,k) Re}
g(k)=\sum_{j=1}^{m-1}\frac{1}{\beta}\Re\left(\frac{\beta-\beta_j}{(m+1)\beta_j-2m}(\beta_j-1)\beta_j^{k}\right)\,,
\end{equation}
and estimate
\begin{multline*}
|g(k)|\leq\sum_{j=1}^{m-1}\frac{1}{\beta}\left|\Re\left(\frac{\beta-\beta_j}{(m+1)\beta_j-2m}(\beta_j-1)\beta_j^{k}\right)\right| \\
\leq\sum_{j=1}^{m-1}\frac{1}{\beta}\left|\frac{\beta-\beta_j}{(m+1)\beta_j-2m}\right|\cdot|\Re(\beta_j)-1|\cdot\left|\beta_j^{k}\right|\,.
\end{multline*}

To finish our proof we will deduce
for all $j=1,\ldots,m-1$,
\begin{equation}\label{zlomek est}
\frac{1}{\beta}\left|\frac{\beta-\beta_j}{(m+1)\beta_j-2m}\right|\leq\frac{1}{2(m-1)}\,.
\end{equation}

Since
\begin{equation}\label{vytknuti}
\frac{1}{\beta}\left|\frac{\beta-\beta_j}{(m+1)\beta_j-2m}\right|=\frac{1}{2(m-1)}\left|\frac{m-1-(m-1)\frac{\beta_j}{\beta}}{m-(m+1)\frac{\beta_j}{2}}\right|\,,
\end{equation}
it suffices to prove that
\[
\left|\frac{m-1-(m-1)\frac{\beta_j}{\beta}}{m-(m+1)\frac{\beta_j}{2}}\right|\leq 1\,.
\]
We have
\begin{equation}\label{squared}
\left|\frac{m-1-(m-1)\frac{\beta_j}{\beta}}{m-(m+1)\frac{\beta_j}{2}}\right|^2=\frac{\left[m-1-\frac{m-1}{\beta}\Re(\beta_j)\right]^2+\left[\frac{m-1}{\beta}\Im(\beta_j)\right]^2}{\left[m-\frac{m+1}{2}\Re(\beta_j)\right]^2+\left[\frac{m+1}{2}\Im(\beta_j)\right]^2}\,.
\end{equation}
 Lemma~\ref{lemma x0} implies $2-\beta<\frac{2}{m+1}<\frac{4}{m+1}$; hence
\begin{equation}\label{hruby odhad beta0}
\frac{m-1}{\beta}<\frac{m+1}{2}\,.
\end{equation}
Therefore
\begin{equation}\label{imaginarni}
\left[\frac{m-1}{\beta}\Im(\beta_j)\right]^2<\left[\frac{m+1}{2}\Im(\beta_j)\right]^2\,.
\end{equation}
In what follows we demonstrate that
\begin{equation}\label{realna}
\left|m-1-\frac{m-1}{\beta}\Re(\beta_j)\right|<\left|m-\frac{m+1}{2}\Re(\beta_j)\right|\,.
\end{equation}
Since $\beta\in(1,2)$ and $|\beta_j|<1$, we have
\[
0<m-1-\frac{m-1}{\beta}\Re(\beta_j)=m-\frac{m+1}{2}\Re(\beta_j)-1+\left(\frac{m+1}{2}-\frac{m-1}{\beta}\right)\Re(\beta_j)\,.
\]
It holds $\Re(\beta_j)<1$, and the expression $\frac{m+1}{2}-\frac{m-1}{\beta}$ is positive due to equation~\ref{hruby odhad beta0}; therefore
\begin{multline*}
-1+\left(\frac{m+1}{2}-\frac{m-1}{\beta}\right)\Re(\beta_j)<-1+\frac{m+1}{2}-\frac{m-1}{\beta}=-(m-1)\left(\frac{1}{\beta}-\frac{1}{2}\right)<0\,.
\end{multline*}
Hence
\[
0<m-1-\frac{m-1}{\beta}\Re(\beta_j)<m-\frac{m+1}{2}\Re(\beta_j)\,,
\]
i.e., \eqref{realna} holds true.
Equation~\eqref{squared} together with inequalities~\eqref{imaginarni} and~\eqref{realna} implies
\[
\left|\frac{m-1-(m-1)\frac{\beta_j}{\beta}}{m-(m+1)\frac{\beta_j}{2}}\right|^2<
\frac{\left[2m-(m+1)\Re(\beta_j)\right]^2+[(m+1)\Im(\beta_j)]^2}{\left[2m-(m+1)\Re(\beta_j)\right]^2+[(m+1)\Im(\beta_j)]^2}=1\,.
\]
\end{proof}

\begin{corollary}
\begin{equation}\label{soucin}
\sum_{k=2m}^{+\infty}|g(k)|\leq\frac{1}{2(m-1)}\sum_{j=1}^{m-1}\cdot\frac{|\Re(\beta_j)-1|}{1-|\beta_j|}\cdot\frac{1}{|2-\beta_j|^2}\,.
\end{equation}
\end{corollary}

\begin{proof} Using \eqref{|jenom g(0,k)|}, we can estimate
\[
    \sum_{k=2m}^{+\infty}|g(k)|\leq\frac{1}{2(m-1)}\sum_{j=1}^{m-1}|\Re(\beta_j)-1|\sum_{k=2m}^{+\infty}|\beta^k_j |= \frac{1}{2(m-1)}\sum_{j=1}^{m-1}|\Re(\beta_j)-1|\frac{|\beta_k|^{2m}}{1-|\beta_j|}\,.
\]
Finally, we use Observation~\ref{obs. beta} to rewrite $|\beta_j|^{2m}=1/|2-\beta_j|^2$.
\end{proof}


At this stage we apply the information on $|\beta_j|$ for $j=1,\ldots,m-1$, derived in Lemma~\ref{Lemma B_j est}.

\begin{proposition}
It holds
\begin{equation}\label{|sum g(0,k)| est}
\sum_{k=2m}^{+\infty}|g(k)|<\frac{1}{2(m-1)}\cdot\frac{1}{1-\frac{2\ln3}{3m}}\left(\frac{2m}{1-\frac{\ln3}{m}}\sum_{j=1}^{m-1}\frac{1-\cos\gamma_j}{(5-4\cos\gamma_j)\ln(5-4\cos\gamma_j)}+\sum_{j=1}^{m-1}\frac{\cos\gamma_j}{5-4\cos\gamma_j}\right)\,.
\end{equation}
\end{proposition}

\begin{proof}
We will estimate summands   from   inequality \eqref{soucin}.
 In the notation $\beta_j=B_j\e^{\I\gamma_j}$, we have
\[
\frac{|\Re(\beta_j)-1|}{1-|\beta_j|}=\frac{1-B_j\cos\gamma_j}{1-B_j}=\frac{1-\cos\gamma_j}{1-B_j}+\cos\gamma_j\,,
\]
thus equation~\eqref{B_j est} \fixE{from
Appendix}  implies
\begin{equation}\label{cinitel1}
\frac{1-\cos\gamma_j}{1-B_j}+\cos\gamma_j\leq\frac{1-\cos\gamma_j}{\ln(5-4\cos\gamma_j)}\cdot\frac{2m}{1-\frac{\ln3}{m}}+\cos\gamma_j\,.
\end{equation}
Concerning the term $1/|2-\beta_j|^2$, it holds
\begin{multline*}
\frac{1}{|2-\beta_j|^2}=\frac{1}{4-4B_j\cos\gamma_j+B_j^2}=\frac{1}{5-4\cos\gamma_j+4(1-B_j)\cos\gamma_j-2(1-B_j)+(1-B_j)^2} \\
<\frac{1}{5-4\cos\gamma_j}\cdot\frac{1}{1-(1-B_j)\frac{2-4\cos\gamma_j}{5-4\cos\gamma_j}}\,.
\end{multline*}
It is easy to see that $\frac{2-4\cos\gamma}{5-4\cos\gamma}\leq\frac{2}{3}$, therefore, it suffices to estimate $1-B_j$ from above. Since
\[
B_j=\frac{1}{\sqrt[2m]{4-4B_j\cos\gamma_j+B_j^2}}>\frac{1}{\sqrt[2m]{9}}=\frac{1}{\sqrt[m]{3}}
\]
and
\[
\sqrt[m]{3}=\e^{\frac{\ln3}{m}}<\left[\left(1+\frac{1}{\frac{m}{\ln3}-1}\right)^\frac{m}{\ln3}\right]^\frac{\ln3}{m}=\frac{\frac{m}{\ln3}}{\frac{m}{\ln3}-1}\,,
\]
it holds $B_j>\frac{\frac{m}{\ln3}-1}{\frac{m}{\ln3}}$. Hence $1-B_j<\frac{\ln3}{m}$ for all $j=1,\ldots,m-1$.
Consequently,
\begin{equation}\label{cinitel2}
\frac{1}{|2-\beta_j|^2}<\frac{1}{5-4\cos\gamma_j}\cdot\frac{1}{1-\frac{2}{3}\cdot\frac{\ln3}{m}}\,.
\end{equation}
Inequality~\eqref{soucin} combined with estimates \eqref{cinitel1} and \eqref{cinitel2} leads to formula~\eqref{|sum g(0,k)| est}.
\end{proof}

The following lemma is an essential component of our calculation. It uses the information on $\gamma_j$ obtained in Lemma~\ref{Lemma gamma_j est}.

\begin{lemma}\label{Lemma f est}
It holds
\begin{equation}\label{f est}
\sum_{j=1}^{m-1}\frac{1-\cos\gamma_j}{(5-4\cos\gamma_j)\ln(5-4\cos\gamma_j)}\leq\frac{m}{2\pi}\int_0^{2\pi}\frac{1-\cos x}{(5-4\cos\gamma_j)\ln(5-4\cos x)}\d x-\frac{1}{6}+\frac{m-1}{m}\cdot\frac{\pi}{16}\left(1+\frac{1}{36}\right)\,.
\end{equation}
\end{lemma}

\begin{proof}
Let us denote
\[
f(x)=\frac{1-\cos x}{\ln(5-4\cos x)}\,;
\]
then
\begin{equation}\label{Riemann}
\sum_{j=1}^{m-1}\frac{1-\cos\gamma_j}{\ln(5-4\cos\gamma_j)}=\frac{m}{2\pi}\sum_{j=1}^{m-1}\frac{2\pi}{m}f(\gamma_j)
\end{equation}
The estimate~\eqref{gamma_j est} implies $\gamma_j\in\left(\frac{2\pi}{m}\left(j-\frac{1}{2}\right),\frac{2\pi}{m}\left(j+\frac{1}{2}\right)\right)$. Therefore, the sum~\eqref{Riemann} is a Riemann sum of the function $f$ with respect to the tagged partition
\begin{equation*}\label{partition}
\frac{\pi}{m}=x_0<x_1<\ldots<x_{m-1}=2\pi-\frac{\pi}{m}\,, \qquad\text{where } x_j=\frac{2\pi}{m}\left(j+\frac{1}{2}\right)\,,
\end{equation*}
of interval $\left[\frac{\pi}{m},2\pi-\frac{\pi}{m}\right]$.
Let us rewrite the summands of~\eqref{Riemann} using a trivial identity
\[
\frac{2\pi}{m}f(\gamma_j)=\int_{x_{j-1}}^{x_{j}}f(x)\d x+\int_{x_{j-1}}^{x_{j}}(f(\gamma_j)-f(x))\d x\,.
\]
Since
\[
f(\gamma_j)-f(x)\leq|x-\gamma_j|\cdot\max_{x\in(x_{j-1},x_j)}\{|f'(x)|\}\leq|x-\gamma_j|\cdot\max_{x\in[0,2\pi)}\{|f'(x)|\}\,,
\]
we have
\[
\frac{2\pi}{m}f(\gamma_j)\leq\int_{x_{j-1}}^{x_{j}}f(x)\d x+\max_{x\in[0,2\pi)}\{|f'(x)|\}\int_{x_{j-1}}^{x_{j}}|x-\gamma_j|\d x\,.
\]
Now we apply another identity, valid for any $\gamma_j\in[x_{j-1},x_j]$,
\[
\int_{x_{j-1}}^{x_{j}}|x-\gamma_j|\d x=\int_{x_{j-1}}^{\gamma_j}(\gamma_j-x)\d x+\int_{\gamma_j}^{x_{j}}(x-\gamma_j)\d x=\int_{0}^{\gamma_j-x_{j-1}}x\d x+\int_{0}^{x_{j}-\gamma_j}x\d x\,,
\]
which provides us, using the estimate~\eqref{gamma_j est}, the inequality
\[
\int_{x_{j-1}}^{x_{j}}|x-\gamma_j|\d x\leq\int_{0}^{\frac{\pi}{m}+\frac{\pi}{6m}}x\d x+\int_{0}^{\frac{\pi}{m}-\frac{\pi}{6m}}x\d x=\frac{\pi^2}{m^2}\left(1+\frac{1}{36}\right)\,.
\]
Hence
\[
\frac{2\pi}{m}f(\gamma_j)\leq\int_{x_{j-1}}^{x_{j}}f(x)\d x+\max_{x\in[0,2\pi)}\{|f'(x)|\}\frac{\pi^2}{m^2}\left(1+\frac{1}{36}\right)\,.
\]
Consequently,
\begin{multline*}
\sum_{j=1}^{m-1}f(\gamma_j)\leq\frac{m}{2\pi}\left(\int_{\frac{\pi}{m}}^{2\pi-\frac{\pi}{m}}f(x)\d x+(m-1)\max_{x\in[0,2\pi)}\{|f'(x)|\}\frac{\pi^2}{m^2}\left(1+\frac{1}{36}\right)\right)\,.
\end{multline*}
Furthermore, it can be checked that $f(x)\geq1/6$ for all $x\in(0,\pi/2)\cup(3\pi/2,2\pi)$ and $\lim_{x\to0}f(x)=1/4>1/6$, hence
\[
\int_{\frac{\pi}{m}}^{2\pi-\frac{\pi}{m}}f(x)\d x=\int_{0}^{2\pi}f(x)\d x-\int_0^{\frac{\pi}{m}}f(x)\d x-\int_{2\pi-\frac{\pi}{m}}^{2\pi}f(x)\d x\leq\int_{0}^{2\pi}f(x)\d x-\frac{2\pi}{m}\cdot\frac{1}{6}\,.
\]
Finally, a numerical calculation gives $\max_{x\in[0,2\pi)}\{|f'(x)|\}<\frac{1}{8}$.
To sum up,
\begin{multline*}
\sum_{j=1}^{m-1}\frac{1-\cos\gamma_j}{(5-4\cos\gamma_j)\ln(5-4\cos\gamma_j)} \\
\leq\frac{m}{2\pi}\left(\int_{0}^{2\pi}\frac{1-\cos x}{(5-4\cos\gamma_j)\ln(5-4\cos x)}\d x-\frac{2\pi}{m}\cdot\frac{1}{6}+(m-1)\frac{1}{8}\cdot\frac{\pi^2}{m^2}\left(1+\frac{1}{36}\right)\right)\,,
\end{multline*}
whence we obtain the sought formula~\eqref{f est}.
\end{proof}

\begin{lemma}
It holds
\begin{equation}\label{cos est}
\sum_{j=1}^{m-1}\frac{\cos\gamma_j}{5-4\cos\gamma_j}\leq\frac{m}{6}+\frac{5}{6}\,.
\end{equation}
\end{lemma}

\begin{proof}
If we define $\gamma_m:=2\pi$, we can write
\[
\sum_{j=1}^{m-1}\frac{\cos\gamma_j}{5-4\cos\gamma_j}=\frac{m}{2\pi}\sum_{j=1}^{m}\frac{2\pi}{m}\frac{\cos\gamma_j}{5-4\cos\gamma_j}-1\,.
\]
The sum $\sum_{j=1}^{m}\frac{2\pi}{m}f(\gamma_j)$ for $f(\gamma):=\frac{\cos\gamma}{5-4\cos\gamma}$ will be calculated in a similar way as in the proof of Lemma~\ref{Lemma f est}.
Namely,
it is a Riemann sum of the function $f$ with respect to the tagged partition
\[
\frac{\pi}{m}=x_0<x_1<\ldots<x_{m-1}<x_m=2\pi+\frac{\pi}{m}\,, \qquad\text{where } x_j=\frac{2\pi}{m}\left(j+\frac{1}{2}\right)\,,
\]
of interval $\left[\frac{\pi}{m},2\pi+\frac{\pi}{m}\right]$.
Following the steps of the proof of Lemma~\ref{Lemma f est}, we obtain
\begin{multline*}
\sum_{j=1}^{m}f(\gamma_j)\leq\frac{m}{2\pi}\left(\int_{\frac{\pi}{m}}^{2\pi+\frac{\pi}{m}}f(x)\d x+(m-1)\max_{x\in[0,2\pi)}\{|f'(x)|\}\frac{\pi^2}{m^2}\left(1+\frac{1}{36}\right)+\frac{\pi^2}{m^2}\max_{x\in[2\pi,2\pi+\pi/m)}\{|f'(x)|\}\right) \\
<\frac{m}{2\pi}\left(\int_{\frac{\pi}{m}}^{2\pi+\frac{\pi}{m}}f(x)\d x+m\cdot\max_{x\in[0,2\pi)}\{|f'(x)|\}\frac{\pi^2}{m^2}\left(1+\frac{1}{36}\right)\right)\,.
\end{multline*}
With regard to the properties of $\cos$, we find
\[
\int_{\frac{\pi}{m}}^{2\pi+\frac{\pi}{m}}\frac{\cos x}{5-4\cos x}\d x=2\int_{0}^{\pi}\frac{\cos x}{5-4\cos x}\d x=2\left[-\frac{x}{4}+\frac{5}{6}\arctan\left(3\tan\frac{x}{2}\right)\right]_0^\pi=\frac{\pi}{3}\,.
\]
Furthermore,
\[
\max_{x\in[0,2\pi)}\{|f'(x)|\}=\frac{5}{2}\cdot\frac{\sqrt{10\sqrt{153}-11}}{(15-\sqrt{153})^2}<\frac{9}{8}\,.
\]
To sum up,
\[
\sum_{j=1}^{m-1}\frac{\cos\gamma_j}{5-4\cos\gamma_j}\leq\frac{m}{2\pi}\left(\frac{\pi}{3}+\frac{9}{8}\cdot\frac{\pi^2}{m}\left(1+\frac{1}{36}\right)\right)-1=\frac{m}{6}+\frac{\pi}{2}\left(1+\frac{1}{36}\right)\frac{9}{8}-1<\frac{m}{6}+\frac{5}{6}\,.
\]
\end{proof}

\begin{proposition}\label{zbytek}
For all $m\geq4$, it holds
\begin{equation}\label{|sum g(0,k)| est2}
\sum_{k=2m}^{+\infty}|g(k)|<\frac{A}{2\pi}m+1\,,
\end{equation}
where
\begin{equation}\label{A}
A:=\int_{0}^{2\pi}\frac{1-\cos x}{(5-4\cos x)\ln(5-4\cos x)}\d x\approx0.909\,.
\end{equation}
\end{proposition}

\begin{proof}
Recall that
\[
\sum_{k=2m}^{+\infty}|g(k)|<\frac{1}{2(m-1)}\cdot\frac{1}{1-\frac{2\ln3}{3m}}\left(\frac{2m}{1-\frac{\ln3}{m}}\sum_{j=1}^{m-1}\frac{1-\cos\gamma_j}{(5-4\cos\gamma_j)\ln(5-4\cos\gamma_j)}+\sum_{j=1}^{m-1}\frac{\cos\gamma_j}{5-4\cos\gamma_j}\right)\,.
\]
cf. formula~\eqref{|sum g(0,k)| est}.
If we estimate the sums using inequalities~\eqref{f est} and~\eqref{cos est}, we obtain
\begin{multline*}
\sum_{k=2m}^{+\infty}|g(k)|-\frac{A}{2\pi}m-1 \\
<
\frac{1}{2(m-1)}\cdot\frac{1}{1-\frac{2\ln3}{3m}}\left(\frac{2m}{1-\frac{\ln3}{m}}\left(\frac{m}{2\pi}A-\frac{1}{6}+\frac{m-1}{m}\cdot\frac{\pi}{16}\left(1+\frac{1}{36}\right)\right)+\frac{m}{6}+\frac{5}{6}\right)
-\frac{A}{2\pi}m-1\,.
\end{multline*}
A numerical integration gives $A\approx0.909\in(0.9,0.91)$. For such value of $A$, the expression above is negative for all $m\geq4$; i.e.,
\[
\sum_{k=0}^{+\infty}|g(k)|-\frac{A}{2\pi}m-1<0 \qquad\text{for all $m\geq4$}\,.
\]
\end{proof}


\section{Main result}\label{Section 6}

\begin{theorem}\label{hlavni}
    For every $m\geq5$, the $m$-bonacci word is $c$-balanced with
\[
c = \lfloor   \kappa m   \rfloor +12,
\]
where $\kappa = \frac{2}{\pi}\int_{0}^{2\pi}\frac{1-\cos x}{(5- 4 \cos x)\ln (5-4\cos x)} {\rm d}x   \approx 0.58$.
\end{theorem}

\begin{proof}

In Propositions \ref{zacatek} and  \ref{zbytek} we showed

\[\sum_{k=0}^{2m-1}|g(k)|<\frac{5}{4}\qquad \hbox{and} \qquad
\sum_{k=2m}^{+\infty}|g(k)|<\frac{A}{2\pi}m+1\quad\text{for all $m\geq4$}\,;\]
therefore,
\begin{equation}\label{spolu}  \sum_{k=0}^{+\infty}|g(0,k)| < \tfrac94 + \frac{A}{2\pi} m\,,\end{equation}
where $A = \int_{0}^{2\pi}\frac{1-\cos x}{(5- 4 \cos x)\ln (5-4\cos x)} {\rm d}x   \approx 0.909$.

\medskip

Having this bound in hand, we can use  Corollary~\ref{bound} to estimate
 the balance constant of letter $0$  by
\[c_0 \leq 2\sum_{k=0}^{+\infty}|g(0,k)| \leq  \tfrac92 + \frac{A}{\pi} m\,.\]

Since  $  \tfrac92 + \frac{A}{\pi} m \leq 2^{m-1} - 3 $ for any $m\geq 5$, the assumption of
Proposition \ref{prop_c0}    is \fixK{fulfilled} and thus \fixK{the} $m$-bonacci word is $c$-balanced with

\[
    c= 2c_0 +3 \leq   3+  4\sum_{k=0}^{+\infty}|g(0,k)| \leq  12 + \frac{2A}{\pi}m \,,
\]
which proves the theorem.

\end{proof}




\section{Acknowledgement}

We acknowledge financial support by the Czech Science Foundation
grant GA\v{C}R 201/09/0584, by the Grant Agency of the Czech
Technical University in Prague, grant SGS11/162/OHK4/3T/14, and by
the Foundation \fixK{"Nad\'{a}n\'{i}} Josefa,
Marie a Zde\v{n}ka Hl\'{a}vkov\'{y}ch".

\bibliographystyle{fundam}
\bibliography{citations}

\nocite{*}
\bibliographystyle{fundam}
\bibliography{figuide}


\appendix
\section{On eigenvalues of $M$}\label{Apendix A}

In this section we examine the eigenvalues of the matrix of
substitution. In particular, we estimate their absolute values and
arguments. Such information is essential for estimating the sums
$\sum_{k=0}^{2m-1}|g(0,k)|$ and $\sum_{k=2m}^{+\infty}|g(0,k)|$ in
Section~\ref{sumace}.

Let us recall that the eigenvalues of the matrix  of substitution
$M$  are zeros of its characteristic  polynomial
$p(x)=x^m-x^{m-1}-x^{m-2}-\fixK{\ldots}-x-1$. The following observation
will make further calculations substantially simpler.

\begin{observation}\label{obs. beta}
Every zero of the p\fixK{o}lynomial  $p(x)$ is a root of the equation
\begin{equation}\label{beta}
x^m(2-x)=1\,.
\end{equation}
\end{observation}

\begin{proof}
 For every $x\neq 1$,  we can write
\begin{equation}\label{p}
p(x)=x^m-\frac{x^m-1}{x-1}=\frac{x^{m+1}-2x^m+1}{x-1}\,.
\end{equation}
In particular, $p(\beta_j)=0$ implies
$\beta_j^{m+1}-2\beta_j^m+1=0$, whence  $\beta_j$ is a root of
equation~\eqref{beta}.
\end{proof}

At first we derive a fine estimate on $\beta$, which is needed
for calculating the sum $\sum_{k=0}^{2m-1}|g(0,k)|$.

\begin{lemma}\label{lemma x0}
The dominant  eigenvalue $\beta>1$ of the matrix of substitution
$M$ obeys the inequalities
\begin{equation}\label{x0}
\frac{1}{2^m-\frac{m}{2}}<2-\beta<\frac{1}{2^m-\frac{m+1}{2}}\,.
\end{equation}
\end{lemma}

\begin{proof}
Observation~\ref{obs. beta} implies $\beta^m(2-\beta)=1$,
hence $\beta<2$. Let us set $x_0:=2-\beta$. Obviously, $x_0$
is a root of the polynomial
\[
q(x)=(2-x)^m\cdot x-1\,.
\]
Since $\beta\in(1,2)$, necessarily $x_0\in(0,1)$. It holds
$q'(x)=(2-x)^{m-1}(2-x-mx)$, therefore, $q$ grows in $[0,2/(m+1)]$
and decreases in $[2/(m+1),1]$. Since $q(0)=-1$ and $q(1)=0$, the
root $x_0$ belongs to the interval $(0,2/(m+1))$, in which $q$
grows. Consequently, proving inequalities~\eqref{x0} consists in
showing that
\[
q\left(\frac{1}{2^m-\frac{m}{2}}\right)<0<q\left(\frac{1}{2^m-\frac{m+1}{2}}\right)\,.
\]
Let us start with the estimate of $2-\beta$ from above. We have
\[
q\left(\frac{1}{2^m-\frac{m+1}{2}}\right)=\left(2-\frac{1}{2^m-\frac{m+1}{2}}\right)^m\frac{1}{2^m-\frac{m+1}{2}}-1=\left(1-\frac{1}{2^{m+1}-(m+1)}\right)^m\frac{1}{1-\frac{m+1}{2^{m+1}}}-1\,.
\]
Since $(1+x)^m>1+mx$ for all $x\in(-1,1)$, it holds
\[
q\left(\frac{1}{2^m-\frac{m+1}{2}}\right)>\frac{1-\frac{m}{2^{m+1}-(m+1)}}{1-\frac{m+1}{2^{m+1}}}-1=\frac{-\frac{m}{2^{m+1}-(m+1)}+\frac{m+1}{2^{m+1}}}{1-\frac{m+1}{2^{m+1}}}=\frac{2^{m+1}-(m+1)^2}{\left[2^{m+1}\left(1-\frac{m+1}{2^{m+1}}\right)\right]^2}\geq0
\]
for all $m\geq3$. Hence, $q\left(1/(2^{m}-\frac{m+1}{2})\right)>0$
for all $m\geq3$. If $m=2$, the statement can be proved in the
same way, just we use the exact expression $(1+x)^2=1+2x+x^2$
instead of the estimate $(1+x)^m>1+mx$.

Let us proceed to the extimate of $2-\beta$ from below.
\[
q\left(\frac{1}{2^{m}-\frac{m}{2}}\right)=\left(2-\frac{1}{2^{m}-\frac{m}{2}}\right)^m\frac{1}{2^{m}-\frac{m}{2}}-1=\frac{1}{1-\frac{m}{2^{m+1}}}\left[\left(1-\frac{1}{2^{m+1}-m}\right)^m-\left(1-\frac{m}{2^{m+1}}\right)\right]\,.
\]
For all $x\in(-1,0)$, it holds $(1+x)^m<1+mx+{m\choose2}x^2$;
therefore,
\begin{multline*}
\left(1-\frac{1}{2^{m+1}-m}\right)^m-\left(1-\frac{m}{2^{m+1}}\right)<1-\frac{m}{2^{m+1}-m}+\frac{m(m-1)}{2(2^{m+1}-m)^2}-1+\frac{m}{2^{m+1}} \\
=\frac{m}{2(2^{m+1}-m)^2}\left(-2^{m+2}+2m+m-1+2^{m+2}-4m+\frac{m^2}{2^m}\right) \\
=\frac{m}{2(2^{m+1}-m)^2}\left(-1-m+\frac{m^2}{2^m}\right)<0
\end{multline*}
for all $m\geq2$. Hence $q\left(1/(2^{m}-\frac{m}{2})\right)<0$\,.
\end{proof}

Now we proceed to the eivenvalues $\beta_j$ for $j=1,\ldots,m-1$.
For the sake of convenience let us set $B_j:=|\beta_j|$ and
$\gamma_j:=\arg(\beta_j)$, i.e.,
\[
\beta_j=B_j\e^{\I\gamma_j} \qquad\text{for all
$j=1,\ldots,m-1$}\,.
\]

\begin{lemma}\label{Lemma B_j est}
It holds
\begin{equation}\label{B_j est}
|\beta_j|<1-\frac{\ln(5-4\cos\gamma_j)}{2m}\left(1-\frac{\ln3}{m}\right)
\end{equation}
for all $j=1,\ldots,m-1$.
\end{lemma}

\begin{proof}
Since the value $\beta_j=B_j\e^{\I\gamma_j}$ is a solution of
equation~\eqref{beta}, necessarily
\[
\left|B_j^m\e^{\I m\gamma_j}(2-B_j\e^{\I\gamma_j})\right|^2=1\,.
\]
Hence
\begin{equation}\label{B_j}
B_j^{2m}\left(4-4B_j\cos\gamma_j+B_j^2\right)=1\,.
\end{equation}
Note that if $m\gg1$, then obviously $B_j\approx1$. Therefore,
equation~\eqref{B_j} can be expressed approximately as
\[
B_j^{2m}\left(4-4\cos\gamma_j+1\right)\approx1 \qquad\text{for
$m\gg1$}\,.
\]
Consequently, for $m\gg1$ we have
\begin{multline}\label{aproximace}
B_j\approx\frac{1}{\sqrt[2m]{5-4\cos\gamma_j}}=\e^{-\frac{\ln(5-4\cos\gamma_j)}{2m}}\approx\left[\left(1+\frac{1}{2m}\right)^\frac{1}{2m}\right]^{-\frac{\ln(5-4\cos\gamma_j)}{2m}} \\
=\left(1+\frac{1}{2m}\right)^{-\ln(5-4\cos\gamma_j)}\approx1-\frac{\ln(5-4\cos\gamma_j)}{2m}\,.
\end{multline}
With regard to this approximation, let us set
\begin{equation}\label{delta_j}
B_j=1-\frac{\ln(5-4\cos\gamma_j)}{2m}(1+\delta_j)\,,
\end{equation}
for all $m$, where $\delta_j$ compensates the error of the approximation~\eqref{aproximace}.
Comparing the statement~\eqref{B_j est} with the definition of
$\delta_j$, we shall prove that
\[
\delta_j>-\frac{\ln3}{m} \qquad\text{for all $j=1,\ldots,m-1$}\,.
\]

We proceed by contradiction. Let there be a $j\in\{1,\ldots,m-1\}$
such that $\delta_j\leq-\frac{\ln3}{m}$. (Note that necessarily
$\delta_j>-1$, because $\beta_j$'s are of moduli less than one.)
For all $x>\alpha>1$, it holds
\[
\frac{1}{\left(1-\frac{\alpha}{x}\right)^x}=\left(1+\frac{\alpha}{x-\alpha}\right)^x=\left[\left(1+\frac{1}{\frac{x}{\alpha}-1}\right)^{\frac{x}{\alpha}-1}\right]^\frac{x}{\frac{x}{\alpha}-1}<\e^\frac{x}{\frac{x}{\alpha}-1}=\left(\e^\alpha\right)^{1+\frac{\alpha}{x-\alpha}}\,.
\]
Since $B_j=1-\frac{\alpha}{x}$ for $x=2m$ and
$\alpha=(1+\delta_j)\ln(5-4\cos\gamma_j)$, we have
\begin{multline*}
\frac{1}{B_j^{2m}}<\left(\e^{(1+\delta_j)\ln(5-4\cos\gamma_j)}\right)^{1+\frac{(1+\delta_j)\ln(5-4\cos\gamma_j)}{2m-(1+\delta_j)\ln(5-4\cos\gamma_j)}}=(5-4\cos\gamma_j)^{(1+\delta_j)\left(1+\frac{(1+\delta_j)\ln(5-4\cos\gamma_j)}{2m-(1+\delta_j)\ln(5-4\cos\gamma_j)}\right)}\,.
\end{multline*}
Our assumption on $\delta_j$ implies $\delta_j<0$, therefore
\[
\frac{(1+\delta_j)\ln(5-4\cos\gamma_j)}{2m-(1+\delta_j)\ln(5-4\cos\gamma_j)}\leq\frac{\ln(5-4\cos\gamma_j)}{2m-\ln(5-4\cos\gamma_j)}\,;
\]
hence
\begin{equation}\label{1/Bj above}
\frac{1}{B_j^{2m}}<(5-4\cos\gamma_j)^{(1+\delta_j)\left(1+\frac{\ln(5-4\cos\gamma_j)}{2m-\ln(5-4\cos\gamma_j)}\right)}\,.
\end{equation}
At the same time we have from equation~\eqref{B_j}
\begin{multline}\label{1/Bj below}
\frac{1}{B_j^{2m}}=4-4B_j\cos\gamma_j+B_j^2=5-4\cos\gamma_j+(1-B_j)(4\cos\gamma_j-2)+(1-B_j)^2 \\
>5-4\cos\gamma_j+(1-B_j)(4\cos\gamma_j-2)\,.
\end{multline}
Putting inequalities~\eqref{1/Bj above} and \eqref{1/Bj below}
together, we get
\begin{equation*}
(5-4\cos\gamma_j)^{(1+\delta_j)\left(1+\frac{\ln(5-4\cos\gamma_j)}{2m-\ln(5-4\cos\gamma_j)}\right)}>5-4\cos\gamma_j+(1-B_j)(4\cos\gamma_j-2)\,;
\end{equation*}
hence
\[
(5-4\cos\gamma_j)^{\delta_j+(1+\delta_j)\frac{(1+\delta_j)\ln(5-4\cos\gamma_j)}{2m-(1+\delta_j)\ln(5-4\cos\gamma_j)}}>1+(1-B_j)\frac{4\cos\gamma_j-2}{5-4\cos\gamma_j}\,.
\]
This gives, with regard to equation~\eqref{delta_j},
\begin{equation}\label{est1}
\e^{\left(\delta_j+(1+\delta_j)\frac{\ln(5-4\cos\gamma_j)}{2m-\ln(5-4\cos\gamma_j)}\right)\ln(5-4\cos\gamma_j)}-1>\frac{\ln(5-4\cos\gamma_j)}{2m}(1+\delta_j)\frac{4\cos\gamma_j-2}{5-4\cos\gamma_j}\,.
\end{equation}
Since
$\delta_j\leq-\frac{\ln9}{2m}\leq-\frac{\ln(5-4\cos\gamma_j)}{2m}$
by assumption, it holds
\[
\delta_j+(1+\delta_j)\frac{\ln(5-4\cos\gamma_j)}{2m-\ln(5-4\cos\gamma_j)}\leq0\,,
\]
therefore, the exponent in~\eqref{est1} is negative (or zero).
Moreover, a simple analysis of the exponent, using the fact
$\delta_j>-1$, leads to the inequality
\[
\left(\delta_j+(1+\delta_j)\frac{\ln(5-4\cos\gamma_j)}{2m-\ln(5-4\cos\gamma_j)}\right)\ln(5-4\cos\gamma_j)\geq-\ln9
\qquad\text{for all $\gamma_j\in\R$}\,.
\]
The convexity of the exponential function implies
\[
\e^x-1<\frac{\e^{b}-1}{b}x
\]
for all $b<x\leq0$. Therefore, the left hand side of~\eqref{est1}
obeys
\begin{multline*}
\e^{\left(\delta_j+(1+\delta_j)\frac{\ln(5-4\cos\gamma_j)}{2m-\ln(5-4\cos\gamma_j)}\right)\ln(5-4\cos\gamma_j)}-1 \\
<\frac{1-\e^{-\ln9}}{\ln9}\left(\delta_j+(1+\delta_j)\frac{\ln(5-4\cos\gamma_j)}{2m-\ln(5-4\cos\gamma_j)}\right)\ln(5-4\cos\gamma_j) \\
=\frac{8}{9\ln9}\left(\delta_j+(1+\delta_j)\frac{\ln(5-4\cos\gamma_j)}{2m-\ln(5-4\cos\gamma_j)}\right)\ln(5-4\cos\gamma_j)\,.
\end{multline*}
Inequality~\eqref{est1} together with this estimate impl\fixK{y}
\[
\frac{8}{9\ln9}\left(\delta_j+(1+\delta_j)\frac{\ln(5-4\cos\gamma_j)}{2m-\ln(5-4\cos\gamma_j)}\right)\ln(5-4\cos\gamma_j)>\frac{\ln(5-4\cos\gamma_j)}{2m}(1+\delta_j)\frac{4\cos\gamma_j-2}{5-4\cos\gamma_j}\,.
\]
We divide both sides by $\ln(5-4\cos\gamma_j)$, which is allowed
due to $\gamma_j\neq0$ (recall that $\beta_j\notin(0,+\infty)$ for
all $j=1,\ldots,m-1$); hence
\begin{equation}\label{est2}
\delta_j+(1+\delta_j)\frac{\ln(5-4\cos\gamma_j)}{2m-\ln(5-4\cos\gamma_j)}>\frac{9\ln9}{8}\cdot\frac{1+\delta_j}{2m}\cdot\frac{4\cos\gamma_j-2}{5-4\cos\gamma_j}\,.
\end{equation}
For all $\gamma_j\in\R$, $\ln(5-4\cos\gamma_j)\leq\ln9$ and
\[
\frac{4\cos\gamma_j-2}{5-4\cos\gamma_j}=-1+\frac{3}{5-4\cos\gamma_j}\geq-1+\frac{3}{9}=-\frac{2}{3}\,;
\]
therefore, with regard to inequality~\eqref{est2},
\[
\delta_j+(1+\delta_j)\frac{\ln9}{2m-\ln9}>\frac{9\ln9}{8}\cdot\frac{1+\delta_j}{2m}\cdot\frac{-2}{3}=-\frac{3\ln9}{8m}(1+\delta_j)\,.
\]
Consequently,
\[
\left(1+\frac{1}{2m-\ln9}+\frac{3}{8m}\right)\delta_j>-\frac{1}{2m-\ln9}-\frac{3}{8m}\,;
\]
hence
\[
\delta_j\geq-\frac{\frac{1}{2m-\ln9}+\frac{3}{8m}}{1+\frac{1}{2m-\ln9}+\frac{3}{8m}}\,.
\]
This is a contradiction with the assumption
$\delta_j\leq-\frac{\ln3}{m}$, because
\[
-\frac{\ln3}{m}<-\frac{\frac{1}{2m-\ln9}+\frac{3}{8m}}{1+\frac{1}{2m-\ln9}+\frac{3}{8m}}
\qquad\text{for all $m\geq2$}\,.
\]
\end{proof}

\begin{lemma}\label{Lemma gamma_j est}
The arguments of $\beta_j$ satisfy
\begin{equation}\label{gamma_j est}
\gamma_j\in\left(j\frac{2\pi}{m}-\frac{\pi}{6m},j\frac{2\pi}{m}+\frac{\pi}{6m}\right)
\end{equation}
for all $j=1,\ldots,m-1$.
\end{lemma}

\begin{proof}
Equation~\eqref{beta} has $m+1$ solutions, namely $1$, $\beta$
and $\beta_1,\ldots,\beta_{m-1}$. Therefore, it suffices to show
that every sector
\[
\mathcal{S}_j:=\left\{B\e^{\I\gamma}\,\left|\,B>0\,,\,\gamma\in\left(j\frac{2\pi}{m}-\frac{\pi}{6m},j\frac{2\pi}{m}+\frac{\pi}{6m}\right)\right.\right\}
\qquad\text{for $j=1,\ldots,m-1$}
\]
contains exactly one solution of equation~\eqref{beta}.

Let
\[
\beta=B\e^{\I\gamma}
\]
be a solution of~\eqref{beta}, i.e.,
\[
B^m\e^{\I m\gamma}\left(2-B\e^{\I\gamma}\right)=1\,.
\]
Hence
\begin{equation}\label{2k.pi}
m\gamma=-\arg\left(2-B\e^{\I\gamma}\right)+2j\pi \qquad\text{for a
certain $j\in\Z$}\,.
\end{equation}
We can obviously assume $j\in\{0,1,\ldots,m-1\}$ without loss of generality. Since the solutions $1$ and $\beta$ of equation~\eqref{2k.pi} are obtained for $\gamma=0$, and, therefore, for $j=0$, 
we prove the statement in two steps: 1. We demonstrate that
equation~\eqref{arctan} has exactly one solution for every
$j=1,\ldots,m-1$. 2. We show that the solution corresponding to
$j$ belongs to the sector $S_j$ for every $j=1,\ldots,m-1$.

It holds
\[
2-B\e^{\I\gamma}=2-B\cos\gamma-\I B\sin\gamma\,,
\]
hence
\[
\tan(\arg(2-B\e^{\I\gamma}))=\frac{-B\sin\gamma}{2-B\cos\gamma}=\frac{-\sin\gamma}{\frac{2}{B}-\cos\gamma}\,.
\]
Furthermore, $B<1$ implies $2-B\cos\gamma>0$, hence
\begin{equation}\label{arg()}
\arg(2-B\e^{\I\gamma})\in(-\pi/2,\pi/2)\,,
\end{equation}
i.e., we can write
\[
\arg\left(2-B\e^{\I\gamma}\right)=\arctan\frac{-\sin\gamma}{\frac{2}{B}-\cos\gamma}\,.
\]
To sum up, equation~\eqref{2k.pi} is equivalent to
\begin{equation}\label{arctan}
m\gamma-\arctan\frac{\sin\gamma}{\frac{2}{B}-\cos\gamma}=2j\pi\,.
\end{equation}
For every $j=1,\ldots,m-1$, the left hand side
$L(\gamma)=m\gamma-\arctan\frac{\sin\gamma}{\frac{2}{B}-\cos\gamma}$
of equation~\eqref{arctan}, regarded as a function of $\gamma$
with a fixed $B<1$, is continuous and satisfies
\[
0=L(0) < 2j\pi < 2m\pi= L(2\pi)\,.
\]
Also, a simple calculation gives
\[
L'(\gamma)=m-\frac{\frac{2}{B}\cos\gamma-1}{\left(\frac{2}{B}\right)^2-2\cdot\frac{2}{B}\cos\gamma+1}
>m-\frac{1}{\frac{2}{B}-1}>m-1>0\,.
\]
Consequently, equation~\eqref{arctan} has indeed exactly one
solution for every $j=1,\ldots,m-1$. The solution satisfies
$m\gamma-2j\pi\in\left(-\pi/2,\pi/2\right)$. With regard to the
numbering~\eqref{numbering}, we conclude that
\[
\gamma_j\in\left(\frac{2j\pi}{m}-\frac{\pi}{2m},\frac{2j\pi}{m}+\frac{\pi}{2m}\right)\,.
\]
Now we improve this estimate in order to prove
$\gamma_j\in\mathcal{S}_j$. Since $2/B_j>2$ for all
$j=1,\ldots,m-1$, we have
\[
\left|\frac{-\sin\gamma_j}{\frac{2}{B_j}-\cos\gamma_j}\right|\leq\left|\frac{\sin\gamma_j}{2-\cos\gamma_j}\right|\,.
\]
It is easy to show that
\[
\left|\frac{\sin\gamma}{2-\cos\gamma}\right|\leq\frac{1}{\sqrt{3}}
\qquad\text{for all $\gamma\in\R$}\,,
\]
hence
\begin{equation}\label{pi/6}
\left|\arctan\frac{\sin\gamma_j}{\frac{2}{B_j}-\cos\gamma_j}\right|\leq\arctan\frac{1}{\sqrt{3}}=\frac{\pi}{6}\,.
\end{equation}
By substituting estimate~\eqref{pi/6} into
equation~\eqref{arctan}, we obtain statement~\eqref{gamma_j est}.
\end{proof}

\end{document}